%% file: christian_mugisho_zagabe_alexandre_mauroy_arxiv_new_second.tex
\documentclass[review,onefignum,onetabnum]{siamart190516}

\usepackage{color}
\usepackage{tikz}
\usetikzlibrary{arrows,calc,through,backgrounds,matrix,decorations.pathmorphing}

\newtheorem{example}{Example}
\newtheorem{assumption}{Assumption}
\usepackage{multirow}
\usepackage{multicol}

\nolinenumbers

\everymath{\displaystyle}
\usepackage{subfigure}

\input{christian_shared}

\ifpdf
\hypersetup{
  pdftitle={Uniform global stability of switched nonlinear systems in the Koopman operator framework},
  pdfauthor={C. M. Zagabe and A. Mauroy}
}
\fi

\begin{document}
	
	\maketitle
	
	\begin{abstract}
		
		In this paper, we provide a novel solution to an open problem on the global uniform stability of switched nonlinear systems. Our results are based on the Koopman operator approach and, to our knowledge, this is the first theoretical contribution to an open problem within that framework. By focusing on the adjoint of the Koopman  generator in the Hardy space on the polydisk (or on the real hypercube), we define equivalent linear (but infinite-dimensional) switched systems and we construct a common Lyapunov functional for those systems, under a solvability condition of the Lie algebra generated by the linearized vector fields. A common Lyapunov function for the original switched nonlinear systems is derived from the Lyapunov functional by exploiting the reproducing kernel property of the Hardy space. The Lyapunov function is shown to converge in a bounded region of the state space, which proves global uniform stability of specific switched nonlinear systems on bounded invariant sets.
	\end{abstract}
	
	
	
	
	\section{Introduction}

	Switched systems are hybrid-type models encountered in applications where the dynamics abruptly jump from one behavior to another. They are typically described by a family of subsystems that alternate according to a given commutation law. Stability properties of switched systems have been the focus of intense research effort (see e.g. \cite{shorten2007stability} for a review). In this context, a natural question is whether a switched system with an equilibrium point is \emph{uniformly stable}, that is, stable for any commutation law. It turned out that the uniform stability problem is counter-intuitive and challenging. In the linear case, it is well-known that stable subsystems may induce an unstable switched system. However, uniform stability is guaranteed if the matrices associated with the subsystems are stable and commute pairwise \cite{narendra1994common}, a result which is extended in \cite{liberzon1999stability} to subsystems described by stable matrices generating a solvable Lie algebra. This latter result can be explained by the well-known equivalence between solvable Lie algebra of matrices and the existence of a common invariant flag for those matrices, which allows to construct a common Lyapunov function for the subsystems \cite{zagabe2021switched}.

	In the case of switched nonlinear systems, an open problem was posed in \cite{liberzon2004lie} on the relevance of Lie-algebraic conditions of  vector fields for global uniform stability. Partial solutions have been proposed in this context. It was proven in \cite{mancilla2000condition} that uniform stability holds if the vector fields are individually stable and commute, in which case a common Lyapunov function can be constructed \cite{Shim2001CommonLF,vu2005common}. Uniform stability was also shown for a pair of vector fields generating a third-order nilpotent Lie algebra \cite{sharon2005third} and for particular $r$-order nilpotent Lie algebras \cite{margaliot2006lie}. 
	However, no result has been obtained, which solely relies on the more general solvability property of Lie algebras of the subsystems vector fields. See \cite{liberzon2023commutation} for a historical review about this open problem.
	
	In this paper, we provide a partial solution to the problem introduced in \cite{liberzon2004lie} by proving global uniform stability results for switched nonlinear systems under a general solvability property of Lie algebras. To do so, we rely on the Koopman operator framework \cite{budivsic2012applied, mauroy2020koopman}: we depart from the classical pointwise description of dynamical systems and consider instead the evolution of observable functions. These functions are first assumed to belong to general reproducing kernel Hilbert spaces, and then to the Hardy space of holomorphic functions defined on the complex polydisk or on the real hypercube. This approach is related to the previous works \cite{ikeda2022koopman,rosenfeld2022dynamic,rosenfeld2019occupation,rosenfeld2019occupation2} studying the Koopman operator defined on those spaces. Through this framework, equivalent infinite-dimensional dynamics are generated by linear Koopman generators, so that nonlinear systems are represented by Koopman linear systems that are amenable to global stability analysis \cite{mauroy2016global}. In particular, building on preliminary results obtained in \cite{zagabe2021switched}, we construct a common Lyapunov functional for switched Koopman linear systems. A key point is to focus on the adjoint of the Koopman generators and notice that these operators have a common invariant maximal flag if the linear parts of the subsystems generate a solvable Lie algebra, a condition that is milder than the original assumption proposed in \cite{liberzon2004lie}. Finally, we derive a common Lyapunov function under the form of a convergent infinite series for the original switched nonlinear system. This allows us to obtain a bounded invariant region of the unit polydisk where the switched nonlinear system is globally uniformly asymptotically stable. Moreover, the common Lyapunov function series can be truncated and considered as a candidate common Lyapunov function, which can potentially lead to an approximation of the region of attraction beyond the polydisk. To our knowledge, this is the first time that a novel solution to an open theoretical problem is obtained within the Koopman operator framework.
	
	The rest of the paper is organized as follows. In Section \ref{sec:prelim}, we present some preliminary notions on uniform stability of switched nonlinear systems and give a general introduction to the Koopman operator framework, as well as some specific properties on reproducing kernel Hilbert spaces of analytic functions and, in particular, on (weighted)  Hardy spaces on the polydisk (or on the real hypercube). In Section \ref{sec:mainresult}, we state and prove our main result. We recast the open problem given in \cite{liberzon2004lie} in terms of the existence of an invariant maximal flag and we provide a constructive proof for the existence of a common Lyapunov function. Additional corollaries are also given, which focus on specific classes of vector fields. Our main results are illustrated with two examples in Section \ref{sec:illustra}. Finally, concluding remarks and perspectives are given in Section \ref{sec:concl}.

	\subsection*{Notations}
	We will use the following notations throughout the manuscript.
	For multi-index notations  $\alpha = (\alpha_1 ,..., \alpha_n)\in \mathbb{N}^n$, we define $\vert\alpha\vert = \alpha_1+ \cdots + \alpha_n$ and $z^\alpha = z_1^{\alpha_1}\cdots z_n^{\alpha_n}$.
	The complex conjugate and real part of a complex number $a$ are denoted by $\bar{a}$ and $\Re(a)$, respectively. The transpose-conjugate of a matrix (or vector) $A$ is denoted by $A^\dagger$. The Jacobian matrix of the vector field $F$ at $x$ is given by $JF(x)$.
	The complex polydisk centered at $0$ and of radius $\rho$ is defined by 
	\[\mathbb{D}^n(\rho)=\left\lbrace z\in \mathbb{C}^n: | z_1|<\rho,\cdots, | z_n|<\rho  \right\rbrace.\]
	The real hypercube with edge of length $2\rho$ and centered at $0$ is defined by 
	\[\mathbb{X}^n(\rho)=\left\lbrace x\in \mathbb{R}^n:  |x_1|<\rho,\cdots,  |x_n|<\rho \right\rbrace.\]
	In particular, $\mathbb{D}^n$ (resp. $\mathbb{X}^n$) denotes the unit (i.e. with $\rho=1$) polydisk (resp. hypercube)  and $\partial \mathbb{D}^n$ (resp.  $\partial \mathbb{X}^n$) is its boundary. Similarly, \[\mathbb{B}^n(\rho)=\left\lbrace z\in \mathbb{C}^n: \| z\|<\rho \right\rbrace\] is the ball centered at $0$ and of radius $\rho$. In particular, $\mathbb{B}^n$ is the unit ball. Finally, the floor of a real number is denoted by $\lfloor x\rfloor$.

	
	\section{Preliminaries}
	\label{sec:prelim}
	In this section, we introduce preliminary notions and results on the stability theory for switched systems and on the Koopman operator framework.  
	
	\subsection{Stability of switched systems}
	
	We focus on the uniform asymptotic stability property of switched systems and on the existence of a common Lyapunov function. Some existing results that connect these two main concepts are presented in both linear and nonlinear cases.
	
	\begin{definition}[Switched system]
		A switched system $\dot{x}=F^{(\sigma)}(x)$ is a (finite) set of subsystems 
		\begin{equation}\label{eq:switch}
			\left\lbrace \dot{x}=F^{(i)}(x), \, x\in X\subset \mathbb{R}^n\right\rbrace_{i=1}^m
		\end{equation}
		associated with a commutation law $\sigma :\mathbb{R}^+\rightarrow \left\lbrace 1,\cdots,m\right\rbrace $ indicating which subsystem is activated at a given time. 
	\end{definition}
	In this paper, we make the following standing assumptions.
	\begin{assumption}
		\label{assump1}
		The commutation law $\sigma$ is a piecewise constant function with a finite number of discontinuities on every bounded time interval (see e.g. \cite{liberzon2003switching}).
	\end{assumption}
	
	\begin{assumption}
		\label{assump_hyperbolic}
		The vector fields $F^{(i)}$ admit on $X$ a unique hyperbolic stable equilibrium at $x_e=0\in X$ (without loss of generality), i.e. $F^{(i)}(x_e)=0$ and the eigenvalues $\tilde{\lambda}^{(i)}_j$ of the Jacobian matrices $JF^{(i)}(x_e)$ satisfy $\Re\left(\tilde{\lambda}^{(i)}_j\right)<0$.
	\end{assumption}
	
	\subsubsection{Uniform stability}
	
	According to \cite{liberzon1999basic}, stability analysis of switched systems revolves around three important problems:
	\begin{itemize}
		\item decide whether an equilibrium is stable under the action of the switched system for \emph{any} commutation law $\sigma$, in which case the equilibrium is said to be \emph{uniformly stable},
		\item identify the commutation laws for which the equilibrium is stable, and
		\item construct the commutation law for which the equilibrium is stable.
	\end{itemize}
	In this paper we focus on the first problem related to uniform stability.
	\begin{definition}[Uniform stability]
		Assume that $F^{(i)}(x_e)=0$ for all $i=1,\dots,m$. The equilibrium $x_e=0$ is 
		\begin{itemize}
			\item \emph{uniformly asymptotically stable} (UAS) if $\forall \epsilon >0$, \mbox{$\exists \delta >0$} such that
			\begin{equation*}
				\| x(0)-x_e\| \leq \delta \, \Rightarrow  \,\| x(t)-x_e\| \leq \epsilon,\, \, \forall t>0, \, \forall \sigma
			\end{equation*} 
			and $\forall \epsilon>0$, $\exists \delta >0, T>0$ such that
			\begin{equation*}
				\| x(0)-x_e\| \leq \delta, \,t>T \Rightarrow \| x(t) -x_e\|\leq \epsilon , \, \forall \sigma,
			\end{equation*}
			
			\item \emph{globally uniformly asymptotically  stable} (GUAS) on $ X\subseteq \mathbb{R}^n$ if it is UAS and $\forall \epsilon>0$, $\exists T>0$ such that 
			\begin{equation*}
				x(0)\in X, \, t>T \Rightarrow \| x(t) -x_e\|\leq \epsilon , \, \forall \sigma,
			\end{equation*}
			
			\item \emph{globally uniformly exponentially stable} (GUES) on \mbox{$X\subseteq \mathbb{R}^n$} if $\exists \beta, \lambda >0$ such that
			\begin{equation*}
				x(0)\in X \Rightarrow \| x(t)-x_e\| \leq \beta \|x(0)-x_e\| e^{-\lambda t},\,\, \forall t>0, \forall \sigma.
			\end{equation*}
		\end{itemize}
	\end{definition}
	This definition implies that the subsystems share a common equilibrium. Moreover, a necessary condition is that this equilibrium is asymptotically stable with respect to the dynamics of all individual subsystems. However, this condition is not sufficient, since the switched system might be unstable for a specific switching law. A sufficient condition for uniform asymptotic stability is the existence of a \emph{common Lyapunov function} (CLF).
	
	\begin{definition}[Common Lyapunov function \cite{liberzon2003switching}]
		A positive $\mathcal{C}^{1}$- function
		$V: X \subseteq \mathbb{R}^{n}\rightarrow \mathbb{R}$ is a \emph{common Lyapunov function} on $X \subseteq \mathbb{R}^n$ for the family of subsystems (\ref{eq:switch}) if 
		\begin{equation*}
			\nabla V \cdot F^{(i)}(x)<0 \quad \forall x \in X\setminus\{x_e\}, \quad \forall i=1,\dots,m.
		\end{equation*}
	\end{definition}
	\smallskip
	For switched systems with a finite number of subsystems, a converse Lyapunov result also holds (\cite{liberzon2003switching}, \cite{mancilla2000condition}).
	\begin{theorem}[\cite{mancilla2000condition}]
		Suppose that $X \subseteq \mathbb{R}^n$ is compact and forward-invariant with respect to the flow induced by the subsystems (\ref{eq:switch}). The switched system (\ref{eq:switch}) is GUAS on $X$ if and only if all subsystems share a CLF on $X$. 
	\end{theorem}

	\subsubsection{Lie-algebraic conditions in the linear case}
	
	In the case of switched linear systems $\{\dot{x}=A^{(i)} x, A^{(i)}\in \mathbb{C}^{n\times n}\}_{i=1}^m$, several results related to uniform stability have been proved (see \cite{shorten2007stability} for a review). We focus here on specific results based on Lie-algebraic conditions.
	
	Let $\mathfrak{g}=\textrm{span}\left\lbrace A^{(i)}\right\rbrace_{Lie}$ denote the Lie algebra generated by the matrices $A^{(i)}$, with $i=1,\cdots, m$, and equipped with the Lie bracket $[A^{(i)},A^{(j)}]=A^{(i)}A^{(j)}-A^{(j)}A^{(i)}$. 
	
	\begin{definition}[Solvable  Lie algebra]
		A Lie algebra $\mathfrak{g}$ equipped with the Lie bracket $[.,.]$ is said to be solvable if there exists $k \in \mathbb{N}$ such that $\mathfrak{g}^k=0$, 
		where $\{\mathfrak{g}^j\}_{j \in \mathbb{N}^*}$ is a descendant sequence of ideals defined by
		$$\begin{cases} \mathfrak{g}^1:= \mathfrak{g}\\ \mathfrak{g}^{j+1}:=\left[ \mathfrak{g}^j,\mathfrak{g}^j\right].\end{cases}$$
	\end{definition}
	A general Lie-algebraic criterion for uniform exponential (asymptotic) stability of switched linear systems is given in the following theorem.
	\begin{theorem}[\cite{liberzon1999stability}]\label{thm:lib}
		If all matrices $A^{(i)}$, $i=1,\cdots,m$, are stable (i.e. with eigenvalues $\tilde{\lambda}_{j}^{(i)}$ such that $\Re\left(\tilde{\lambda}_{j}^{(i)}\right)<0$) and if the Lie algebra $\mathfrak{g} $ is solvable, then the switched linear system $\{\dot{x}=A^{(i)} x\}_{i=1}^m$ is GUES.
	\end{theorem}
	As shown in \cite{mori1997solution,shorten1998stability}, this result follows from the simultaneous triangularization of the matrices $A^{(i)}$, which is a well-known property of solvable Lie algebras (see Lie's theorem \ref{thm:lie} in Appendix \ref{sec:appendix}). 
	This property is in fact equivalent to the existence of a \emph{common invariant flag} for complex matrices \cite{dubi2009algorithmic}.
	\begin{definition}[Invariant flag]
		An invariant maximal flag of the set of matrices $ \{A^{(i)}\}_{i=1}^m$ is a set of subspaces $\{S_j\}_{j=1}^n \subseteq \mathbb{C}^n$ such that (i) $A^{(i)} S_j\subset S_j$ for all $i,j$, (ii) $\textrm{dim}(S_j)=j$ for all $j$, and (iii) $S_j\subset S_{j+1}$ for all $j<n$.
	\end{definition}
	The subspace $S_j$ can be described through an orthonormal basis $(v_1,\cdots, v_n)$, so that $S_j=\textrm{span}\left\lbrace v_1,\cdots,v_j\right\rbrace$. Note that the vector $v_1$ is a common eigenvector of the matrices $A^{(i)}$. This basis can be used to construct a CLF.
	\begin{proposition}[\cite{zagabe2021switched}]
		\label{prop:clf}
		Let 
		\begin{equation}\label{eq:syst1}
			\left\{\dot{x}=A^{(i)} \,x, \,\, A^{(i)}\in \mathbb{C}^{n\times n}, \, x\in\mathbb{C}^{n} \right\}_{i=1}^m
		\end{equation}
		be a switched linear system. Suppose that all matrices $A^{(i)}$ are stable and admit a common invariant maximal flag
		$$\left\lbrace 0\right\rbrace \subset S_{1} \subset \cdots\subset S_{n}=\mathbb{C}^n, \quad S_j=\textrm{span}\{v_1,\dots,v_j\}.$$
		Then there exist $\epsilon_j >0$, $j=1,\dots,n$, such that 
		\begin{equation}
			\label{eq:Lyap}
			V(x)=\sum_{j=1}^n \epsilon_j \vert v_j^\dagger x\vert^2
		\end{equation}
		is a CLF for (\ref{eq:syst1}).
	\end{proposition}
	The values $\epsilon_j$ must satisfy the condition
	\begin{equation}
		\label{eq:valmax}
		\epsilon_j > \max_{\substack{i\in\{1,\dots,m\} \\ k\in\{1,\dots,j-1\}}} \epsilon_k \, \dfrac{(n-1)^2}{4}\dfrac{\left| v_k^\dagger A^{(i)} v_j\right|^2}{ \left| \Re\left( \tilde{\lambda}_{j}^{(i)}\right)\right| \left|  \Re\left(\tilde{\lambda}_{k}^{(i)}\right)\right| }
	\end{equation}
	where $\tilde{\lambda}_{j}^{(i)}$ are the eigenvalues of $A^{(i)}$. They can be obtained iteratively from an arbitrary value $\epsilon_1>0$. The geometric approach followed in \cite{zagabe2021switched} provides a constructive way to obtain a CLF, a result that we will leverage in an infinite-dimensional setting for switched nonlinear systems.
	
	\subsubsection{Lie-algebraic condition in the nonlinear case}
	
	In the context of switched nonlinear systems, one has to consider the Lie algebra of vector fields
	\begin{equation}\label{Lie_algebra_vec}
		\mathfrak{g}_F=\textrm{span}\left\lbrace F^{(i)}, i=1,\dots,m\right\rbrace_\textrm{Lie}
	\end{equation}
	equipped with the Lie bracket
	\begin{equation}
		\label{eq:liebracket}
		[F^{(i)},F^{(j)}](x)=JF^{(j)}(x) \, F^{(i)}(x)-JF^{(i)}(x) \, F^{(j)}(x).
	\end{equation}
	
	It has been conjectured in \cite{liberzon2004lie} that Lie-algebraic conditions on \eqref{Lie_algebra_vec} could be used to characterize uniform stability. This problem has been solved partially in \cite{sharon2005third} for third-order nilpotent Lie algebras and in \cite{margaliot2006lie} for particular $r$-order nilpotent Lie algebras. Another step toward more general Lie-algebraic conditions based on solvability has been made in \cite{zagabe2021switched}, a preliminary result that relies on the so-called Koopman operator framework. However, the results obtained in \cite{zagabe2021switched} are restricted to specific switched nonlinear systems that can be represented as finite-dimensional linear ones. In this paper, we build on this preliminary work, further exploiting the Koopman operator framework to obtain general conditions that characterize the GUAS property of switched nonlinear systems.
	
	\subsection{Koopman operator approach to dynamical systems}
	\label{sec:Koopman}
	
	In this section, we present the Koopman operator framework, which is key to extend the result of Proposition \ref{prop:clf} to switched nonlinear systems. We introduce the Koopman semigroup along with its Koopman generator, cast the framework in the context of Lie groups, and describe the finite-dimensional approximation of the operator. 
	
	\subsubsection{Koopman operator}
	
	Consider a continuous-time dynamical system
	\begin{equation}\label{eq:systnon}
		\dot{x} = F(x), \qquad x\in X \subset \mathbb{R}^n, \quad F \in \mathcal{C}^1
	\end{equation}
	which generates a flow $\varphi_t:X \rightarrow X$, with $t\in \mathbb{R}^+$. The \emph{Koopman operator} is defined on a (Banach) space $\mathcal{F}$ and acts on observables, i.e. functions $f:X \rightarrow \mathbb{R}$, $f\in \mathcal{F}$.
	\begin{definition}[Koopman semigroup \cite{lasota1998chaos}]\label{def:koopman_sg} 
		The semigroup of Koopman operators (in short, Koopman semigroup) is the family of linear operators $\left(U_t\right)_{t\geq 0
		}$ defined by $$U_t:\mathcal{D}(U_t)\rightarrow \mathcal{F}, \quad U_tf=f\circ \varphi_t$$
	\end{definition}
	with the domain $\mathcal{D}(U_t)=\left\lbrace f\in \mathcal{F}: U_t f\in \mathcal{F} \right\rbrace $.
	We can also define the associated Koopman generator.
	\begin{definition}[Koopman generator \cite{lasota1998chaos}]\label{def:koopman_gen} The Koopman generator associated with the vector field \eqref{eq:systnon} is the linear operator
		\begin{equation}
			\label{eq:koopmangenerator}
			L_F: \mathcal{D}(L_F) \rightarrow \mathcal{F}, \quad L_F f:=F \cdot \nabla f
		\end{equation}
		with the domain $\mathcal{D}(L_F)=\left\lbrace f\in \mathcal{F}: F \cdot \nabla f\in \mathcal{F} \right\rbrace $.
	\end{definition}
	As shown below (see Lemma \ref{lemmepreuve1}), the Koopman semigroup and the Koopman generators are directly related. When the Koopman semigroup is strongly continuous \cite{engel2000one}, i.e. $\lim_{t\rightarrow 0^+} \| U_t f- f \|_{\mathcal{F}}=0$, the Koopman generator is the infinitesimal generator $L_F f:=\lim_{t\rightarrow 0^+}(U_t f- f)/t$ of the Koopman semigroup. Since the Koopman operator $U_t$ and the generator $L_F$ are both linear, we can describe the dynamics of an observable $f \in \mathcal{D}(L_F)$ through the linear abstract ordinary differential equation
	\begin{equation}\label{eq:systnoninf}
		\dot{f}=L_F f.
	\end{equation}
	
	We can also briefly discuss the spectral properties of the Koopman operator.
	\begin{definition}[Koopman eigenfunction and eigenvalue \cite{budivsic2012applied, mezic2005spectral}]
		An eigenfunction of the Koopman operator is an observable $\phi_{\lambda}\in \mathcal{F}\setminus \left\lbrace 0\right\rbrace $ such that 
		$$L_F \phi_\lambda =\lambda \phi_\lambda.$$
		The value $\lambda \in \mathbb{C}$ is the associated Koopman eigenvalue.
	\end{definition}
	$$U_t \phi_\lambda =e^{\lambda t}\phi_\lambda, \quad \forall t\geq 0.$$
	For a linear system $\dot{x}=Ax$, with $x\in \mathbb{R}^n$, we denote an eigenvalue of $A$ by $\tilde{\lambda}_j$ and its associated left eigenvector by $w_j$. Then $\tilde{\lambda}_j$ is a Koopman eigenvalue and the associated Koopman eigenfunction is given by $\phi_{\tilde{\lambda}_j}(x)=w_j^\dagger x$ \cite{mezic2013analysis}. For a nonlinear system of the form \eqref{eq:systnon} which admits a stable equilibrium $x_e$, the eigenvalues of $JF(x_e)$ are typically Koopman eigenvalues and the associated eigenfunctions are the so-called \emph{principal Koopman eigenfunctions}.
	
\subsubsection{Koopman operator on reproducing kernel Hilbert spaces}
		From this point on, we will define the Koopman operator on reproducing kernel Hilbert spaces (RKHS). In particular cases, we will use the specific (weighted) Hardy space on the polydisk (or on the hypercube). 
		
		For more details, we refer the reader to the work \cite{ikeda2022koopman} where the Koopman operator is studied in general RKHS setting, \cite{rosenfeld2022dynamic,rosenfeld2019occupation,rosenfeld2019occupation2} where the Koopman generator (called Liouville operator) is studied on particular RKHSs. 
		
			\begin{definition}[Reproducing kernel Hilbert space]
				A \textit{reproducing kernel Hilbert space} (RKHS) $\mathcal{H}(X)$ on a set $X$ is a Hilbert space of functions from $X$ to $\mathbb{C}$ or $\mathbb{R}$ such that,  for every $x\in X$, the linear evaluation functional $E_x:\mathcal{H}(X)\longrightarrow \mathbb{C}$ or $\mathbb{R}$, defined by $E_x(f)=f(x)$, is bounded.
			\end{definition}
			A direct consequence of this definition is that, for each $x\in X$, there exists a unique $k_x \in \mathcal{H}(X) $ such that $E_x(\cdot)=\left\langle k_x, \cdot \right\rangle $ (Riesz representation theorem).
			\begin{definition}[Kernel for RKHS]
				The function $k_x$ is called  the \textit{reproducing kernel} for $x$ (or the \textit{evaluation functional} at $x$, with a slight abuse of language) and the function $k:X\times X \longrightarrow \mathbb{C}$ or $\mathbb{R}$, defined by $k(x,y)=k_y(x)$, is called the reproducing kernel for  $\mathcal{H}(X)$.
			\end{definition}

		\subsubsection*{Koopman operator on a general RKHS} 
		In order to define the Koopman operator on the RKHS  $\mathcal{H}(D)$, where  $D\subseteq X$, we will make the following assumption.
			\begin{assumption}\label{assumption:general}$\,$
				\begin{enumerate}
					\item[(i)]	The components $F_\ell,\, \ell=1,\cdots, n$, of the vector field $F$ belong to  $\mathcal{H}(D) $.
					\item[(ii)] There exists an orthonormal basis $\{e_k\}_{k=0}^\infty$ of $\mathcal{H}(D) $ such that \mbox{$ F_\ell\dfrac{\partial e_k }{\partial z_\ell} \in \mathcal{H}(D)$} for all $\ell=1,\cdots, n$ and all $k \in \mathbb{N}$.
					\item[(iii)] The vector field $F$ generates a flow which maps $D$ to $D$ (forward invariance).
				\end{enumerate}
			\end{assumption}
			If the above assumptions are satisfied, the Koopman semigroup $\left(U_t\right)$ and the Koopman generator can be defined according to Definitions \ref{def:koopman_sg} and \ref{def:koopman_gen}, with the space of observable functions $\mathcal{F}=\mathcal{H}(D)$.
			While we assume that the vector field belongs to $\mathcal{H}(D)$, we do not require that it generates a flow that belongs to $\mathcal{H}(D)$. This implies that the Koopman semigroup might not be defined everywhere on $\mathcal{H}(D)$. As shown below, we will only consider its adjoint $U_t^*$ acting on evaluation functionals.

		Now, we recall some important properties that we will use to prove our results.
		\begin{lemma}[General properties of Koopman operator on RKHS]\label{lemmepreuve1}
			Consider a function $f\in \mathcal{H}(D)$, an orthonormal basis $ \{e_k, \, k \in \mathbb{N}\}$ of $\mathcal{H}(D)$  and an evaluation functional $k_z$, with $z\in D$. Assume that Assumption \ref{assumption:general} holds. Then,
			\begin{enumerate}
				\item $ L_F e_k \in \mathcal{H}(D)$ and the domain $\mathcal{D}\left(L_F\right)$ is dense in $\mathcal{H}(D)$,
				\item $ U_t^* k_z = k_{\varphi_t(z)}$,
				\item $\dfrac{d}{dt} \left\langle U_t^* k_z, f \right\rangle = \left\langle L_{F}^*U_t^* k_z, f \right\rangle$.
			\end{enumerate}
		\end{lemma}
		\begin{proof}
			\begin{enumerate}
				\item For all $z\in D$, we have \begin{equation*}
					L_F e_k(z) =F(z) \cdot\nabla e_k (z)
					= \sum_{\ell=1}^n F_\ell(z) \, \dfrac{\partial e_k (z)}{\partial z_\ell}.
				\end{equation*}
				It follows from Assumption \ref{assumption:general} (ii) that 
			\[\left\| L_F e_k\right\|= \left\| \sum_{\ell=1}^n  F_\ell  \, \dfrac{\partial e_k}{\partial z_\ell}\right\|\leq \sum_{\ell=1}^n \left\|  F_\ell  \dfrac{\partial e_k}{\partial z_\ell}\right\|<\infty.\] 
				Moreover $\mathcal{D}\left(L_F\right)$ is dense in $\mathcal{H}(D)$ since  $\{e_k\}_{k=0}^\infty$ is a complete basis. 
				\item  For all $f\in \mathcal{H}(D)$, we have $$ \left\langle U_t^* k_z,f\right\rangle  = \left\langle k_z,U_tf\right\rangle  =\left( U_t f\right) (z)$$ and 
				$$ \left\langle  k_{ \varphi_t (z)},f\right\rangle  = f\left(\varphi_t (z)\right) =\left( U_t f\right) (z),$$
				so that $$  U_t^*k_z=k_{\varphi_t (z)}.$$
				\item  For all $z\in \mathcal{D}$ and all $f\in \mathcal{D}(L_F)$,
				\begin{eqnarray*}
					\dfrac{d}{dt} \left\langle U_t^* k_z, f\right\rangle &=& \dfrac{d}{dt}\left\langle  k_{\varphi_t (z)}, f\right\rangle \\
					&=& \dfrac{d}{dt} f\circ \varphi_t(z) \\
					&=& F\left( \varphi_t(z)\right).\nabla  f\left( \varphi_t(z)\right) \\
					&=& \left\langle   k_{\varphi_t(z)}, L_F f\right\rangle\\
					&=& \left\langle  L_{F}^* U_t^* k_z , f\right\rangle.
				\end{eqnarray*}
				The result follows for all $f$ since $\mathcal{D}(L_F)$ is dense in $\mathcal{H}(D)$.
			\end{enumerate}
		\end{proof}
		In the previous lemma, the second property is a  well-known property of the composition operator on a  RKHS. The third property is also known in the context of strongly continuous semigroup theory (see \cite{engel2000one}).
		Note that, as mentioned in \cite{russo2022liouville} Proposition 2 for spaces of analytic functions, the Koopman generator $L_F$ is densely defined only for a restricted form of the vector field $F$. In our case, this property is guaranteed by Assumption \ref{assumption:general} (ii).

		
		\paragraph*{Koopman infinite matrix} Since $\mathcal{H}(D)$ is assumed to possess an orthonormal basis, it is isomorphic to $l^2$ and the Koopman generator can be represented by the  Koopman infinite matrix
		
		\begin{equation}\label{eq:koopmatrix0}
			\bar{L}_F=\begin{pmatrix}
				\left\langle L_Fe_0,e_0  \right\rangle  & \left\langle L_Fe_1,e_0  \right\rangle  &  \left\langle L_Fe_2,e_0  \right\rangle&  \left\langle L_Fe_3,e_0  \right\rangle &  \cdots \\ 
				\left\langle L_Fe_0,e_1  \right\rangle  & \left\langle L_Fe_1,e_1  \right\rangle  &  \left\langle L_Fe_2,e_1  \right\rangle&  \left\langle L_Fe_3,e_1  \right\rangle &  \cdots \\ 
				\left\langle L_Fe_0,e_2  \right\rangle  &\left\langle L_Fe_1,e_2 \right\rangle  &  \left\langle L_Fe_2,e_2  \right\rangle  &  \left\langle L_Fe_3,e_2  \right\rangle &  \cdots \\ 
				\left\langle L_Fe_0,e_3  \right\rangle  & \left\langle L_Fe_1,e_3  \right\rangle  &  \left\langle L_Fe_2,e_3  \right\rangle &  \left\langle L_Fe_3,e_3  \right\rangle  &  \cdots \\ 
				\left\langle L_Fe_0,e_4  \right\rangle  & \left\langle L_Fe_1,e_4  \right\rangle  &  \left\langle L_Fe_2,e_4  \right\rangle  &  \left\langle L_Fe_3,e_4  \right\rangle &  \cdots \\  
				\vdots & \vdots & \vdots & \vdots & \cdots
			\end{pmatrix},
		\end{equation}
		where the $k$th column contains the components of $L_F e_k$ in the orthonormal basis $\{e_k\}_{k=0}^\infty$. For $f=\sum_{k \in \mathbb{N}} f_k e_k$, we also have that
		\[\langle L_F f, e_j \rangle = \sum_{k \in \mathbb{N}} f_k \langle L_F e_k, e_j \rangle.\]
		\subsubsection*{Koopman operator on a RKHS of analytic functions} 
		
	In this paper, we focus on analytic vector fields, so that we will naturally consider RKHSs of analytic functions on a bounded domain $D \subset \mathbb{C}^n$. In this case, for a proper choice of real weighting coefficients $\omega_\alpha\geq 0,\,  \alpha \in \mathbb{N}^n$, the family of weighted monomials $\{e_\alpha(z)=\sqrt{\omega_\alpha} z^\alpha,\, \alpha \in \mathbb{N}^n\}$ is an orthonormal basis. Famous examples include the \textit{Hardy space} on the polydisk with $\omega_\alpha=1$ (see below), the\textit{ Bergman space} on the polydisk with $\omega_\alpha=|\alpha|+1$, and the \textit{Segal-Bargmann} (or \textit{Fock-Bargmann}) space on  $\mathbb{R}^n$ with $\omega_\alpha=1/\alpha !$. 
			Other examples can be constructed by considering power series on $D$ (see \cite{paulsen2016introduction} Theorem 4.12 or \cite{steinwart2008support} Lemma 4.8).
		
	For RKHS where the orthonormal basis is given by weighted monomials $e_\alpha(z)=\sqrt{\omega_\alpha}z^\alpha$ with $\omega_\alpha \neq 0$, the monomials will be denoted by $e_{k}(z)= \sqrt{\omega_{\alpha(k)}} z^{\alpha(k)}$,
		where the map $\alpha:\mathbb{N} \to \mathbb{N}^n$, $k\mapsto \alpha (k)$ defines a lexicographic order $\prec$ on the monomials, i.e. $e_{k_1}\prec e_{k_2}$ if $|\alpha(k_1)|<|\alpha(k_2)|$, or if $|\alpha(k_1)|=|\alpha(k_2)|$ and $\alpha_j(k_1)>\alpha_j(k_2)$ for the smallest $j$ such that $\alpha_j(k_1)\neq\alpha_j(k_2)$.  In other words, we have $1\prec z_1\prec z_2\prec\cdots\prec z_n$ and, if $e_{k_1}\preceq e_{k_2}$ then $e_{k_1}e_{k_3}\preceq e_{k_2}e_{k_3}$ for any other monomial  $e_{k_3}$. For example, we have for two-dimensional monomials $1\prec z_1\prec z_2\prec z_1^2\prec z_1z_2\prec z_2^2\prec z_1^3\prec z_1^2z_2\prec z_1z_2^2\prec z_2^3\prec z_1^4\prec z_1^3z_2\prec z_1^2z_2^2\prec z_1z_2^3\prec z_2^4\prec \cdots $.

			\begin{remark}
				\label{rem:assmp2}
				In the particular case of RKHSs with an orthonormal basis of weighted monomials,
				Assumption \ref{assumption:general}(ii) directly follows from Assumption \ref{assumption:general}(i) (and Lemma \ref{lemmepreuve1}.1) since the derivatives $\partial e_\alpha /\partial z_\ell$ are monomials belonging to $\mathcal{H}(D)$. This implies that $\partial e_\alpha /\partial z_\ell f \in \mathcal{H}(D)$ for all $f\in \mathcal{H}(D)$, including the case $f=F_l \in \mathcal{H}(D)$.  
		\end{remark}

	The Koopman (or composition) operator has been extensively studied in the Hardy space on the polydisk (see e.g. \cite{rudin1969function,rudin2008function,shapiro2012composition} for more details). This choice of space is well-suited to the case of analytic vector fields that admit a stable hyperbolic equilibrium (see Assumption \ref{assump_hyperbolic}), where it allows to exploit convenient spectral properties of the operator. In particular, the standard orthonormal basis of unweighted monomials can be leveraged in this case and lead to practical corollaries of our main result.
		
		Since we will focus on the Hardy space as a particular case of our main result, it is described below in details. The Hardy space of analytic functions on the polydisk $\mathbb{D}^n$ is the space
		\[\mathbb{H}^2(\mathbb{D}^n)=\left\lbrace f:\mathbb{D}^n\rightarrow \mathbb{C}, \mbox{holomorphic}:\|f\|^2<\infty\right\rbrace, \]
		where 
		\[ \|f\|^2=\lim_{r\rightarrow 1^-} \int_{(\partial \mathbb{D})^n}\vert f\left(r\omega\right)\vert^2dm_n(\omega)\]
		and $m_n$ is the normalized Lebesgue measure on $(\partial \mathbb{D})^n$. 
		The space is equipped with an inner product defined by 
		\[\left\langle f,g\right\rangle=\int_{(\partial \mathbb{D})^n}f\left(\omega\right) \bar g\left(\omega\right) dm_n(\omega),\]
		so that the set of  monomials $\left\lbrace z^\alpha: \, z\in \mathbb{D}^n,\,  \alpha \in  \mathbb{N}^n\right\rbrace$ is a standard orthonormal basis on $\mathbb{H}^2(\mathbb{D}^n)$.
		
		For $f$ and $g$ in $\mathbb{H}^2(\mathbb{D}^n) $, with $f=\sum_{k \in \mathbb{N}} f_k e_k$ and $g=\sum_{k \in \mathbb{N}} g_k e_k$, the isomorphism 
		$$\sum_{k \in \mathbb{N}} f_k \, e_k \mapsto (f_k)_{k \geq 0}$$  
		between $\mathbb{H}^2(\mathbb{D}^n)$ and the $l^2$-space allows to rewrite the norm and the inner product as
		\[\| f\|^2=\sum_{k \in \mathbb{N}} \vert f_k \vert^2 \quad \mbox{and}\quad \left\langle f,g\right\rangle=\sum_{k \in \mathbb{N}}  f_k \, \bar g_k.\]
		
		We also note that $\mathbb{H}^2(\mathbb{D}^n)$ is a reproducing kernel Hilbert space (RKHS) with the  kernel (\cite[Chapter 1]{rudin1969function})
		\begin{equation}\label{eq:cauchy_kernel}
			k\left(z, \xi\right)=\prod_{i=1}^n\dfrac{1}{1-\bar{z}_i \xi_i},\, \, z,\xi \in  \mathbb{D}^n.
		\end{equation}
		It follows that one can define the (bounded) \emph{evaluation functional}
		$f(z)=\left\langle f, k_z\right\rangle 
		$
		with $k_z( \omega)=k\left(z, \omega\right)$.

		Let $\mathbb{X}^n$ be the (real) subset of $\mathbb{D}^n$, i.e., $\mathbb{X}^n$ is the interior of the $n$-(hyper)cube with edges of length $2$. One can define the Hardy space of functions $f \in \mathbb{D}^n$ restricted to $\mathbb{X}^n$ (here the restriction is obtained by considering the real part of $z\in\mathbb{D}^n$), which we denote by $\mathbb{H}^2(\mathbb{X}^n)$. In this case, the norm, inner product, orthonormal basis, and kernel defined above on $\mathbb{H}^2(\mathbb{D}^n)$ are still valid on $\mathbb{H}^2(\mathbb{X}^n)$ (see e.g. \cite[Corollary 5.8]{paulsen2016introduction} or \cite[Lemma 4.3]{steinwart2008support}  for this kind of construction of RKHS).

		\begin{remark} In the case of the Hardy space, Assumptions \ref{assump_hyperbolic} and \ref{assumption:general} are imposed on the unit polydisk (or hypercube) without loss of generality. They could be imposed on a larger or smaller polydisk $\mathbb{D}^n(\rho)$ (or hypercube $\mathbb{X}^n(\rho)$), with $\rho\neq 1$, and then re-expressed on the unit polydisk (or hypercube) through a proper rescaling of the state variables. The same is true for the stability results derived throughout the paper, which are also valid on the unit polydisk without loss of generality.
		\end{remark}

		When the vector field generates a holomorphic flow that is invariant in $\mathbb{D}^n$ (or $\mathbb{X}^n$), we can define the Koopman semigroup on $\mathbb{H}^2(\mathbb{D}^n)$ (or $\mathbb{H}^2(\mathbb{X}^n)$).
		However, we will not impose here the flow to be holomorphic on $\mathbb{D}^n$, mainly because this is a strong condition that restricts the form of the vector field (see the works \cite{chen2016parametric, contreras2011semigroups}), and therefore the applicability of our results. In addition, according to \cite{russo2022liouville}, strong properties of the Koopman generator and semigroup (e.g. boundedness, compactness, etc.) cannot be assumed, unless the dynamics is linear.
			However, as shown in Remark \ref{rem:assmp2}, the Koopman generator $L_F$ is a densely defined operator on  $\mathbb{H}^2(\mathbb{D}^n)$.  Furthermore, since the Hardy space  is composed by analytic functions on $\mathbb{D}^n$ and the domain of the Koopman operator $L_F$ is nontrivial by Assumption \ref{assumption:general}, then $L_F$  is a closed operator on  $\mathbb{H}^2(\mathbb{D}^n)$ (see \cite{rosenfeld2019occupation} Theorem  4.2). 
		
	\paragraph*{Koopman infinite matrix} Now, we derive the form of the Koopman matrix in the case of a basis of monomials. We first note that $L_F e_0=0$ so that the first row and column of the Koopman matrix $\bar{L}_F$ contain only zero entries. By removing them, one obtains the representation of the restriction of the Koopman generator to the subspace of functions $f$ that satisfy $f(0)=0$. This subspace is spanned by the basis $(e_k)_{k\geq 1}$. Note that $k_z-k_0$ belongs to this subspace, since \eqref{eq:cauchy_kernel} implies that $k_z(0)-k_0(0)=0$.
		
		The entries of the Koopman matrix \eqref{eq:koopmatrix0} can be expressed in term of the coefficients of the vector field $F$. Indeed, for $F_\ell(z) = \sum_{\vert \beta \vert\geq 1} a_{\ell,\beta} z^\beta$, the action of the Koopman operator on a monomial is given by 
		\begin{eqnarray}
			L_F z^\alpha &=& \sum_{\ell=1}^n F_\ell(z) \, \alpha_\ell \, z^{(\alpha_1,\dots,\alpha_{\ell-1}, \alpha_\ell-1,\alpha_{\ell+1},\dots,\alpha_n)}\nonumber\\
			&=& \sum_{\ell=1}^n \sum_{\vert \beta \vert\geq 1} a_{\ell,\beta} z^\beta \, \alpha_\ell \, z^{(\alpha_1,\dots,\alpha_{\ell-1}, \alpha_\ell-1,\alpha_{\ell+1},\dots,\alpha_n)}\nonumber\\
			&=&  \sum_{\ell=1}^n \alpha_\ell \left( \sum_{ \left(\beta_1,\dots,\beta_n\right)\in \mathbb{N}^n } a_{\ell,\left(\beta_1,\dots,\beta_n\right)}   z^{(\beta_1+\alpha_1,\dots, \beta_\ell+\alpha_\ell-1,\beta_{\ell+1}+\alpha_{\ell+1},\dots,\beta_n+\alpha_n)}\right).\nonumber   
		\end{eqnarray}
		By setting $ \gamma_1=\beta_1+\alpha_1, \dots ,\gamma_\ell=\beta_\ell+\alpha_\ell-1 ,\dots, \gamma_n=\beta_n+\alpha_n$, we obtain
		\begin{eqnarray}\label{eq:koopman_on_monom}
			L_F z^\alpha   
			&=& \sum_{\ell=1}^n \alpha_\ell \sum_{ |\gamma |\geq |\alpha|} a_{\ell,\left(\gamma_1-\alpha_1,\dots,\gamma_\ell-\alpha_\ell+1,\dots,\gamma_n-\alpha_n\right)}   z^{(\gamma_1,\gamma_2,\dots,\gamma_n)}\nonumber\\
			& =& \sum_{\ell=1}^n \alpha_\ell \sum_{ |\gamma |\geq |\alpha|} a_{\ell,(\gamma-\alpha)_\ell}   z^{\gamma},
		\end{eqnarray}
		where we assume by convention that $a_{\ell,\alpha}=0$ if $\alpha$ contains a negative component, and where we denote 
		\begin{equation}
			\label{eq_not_index_l}
			(\gamma-\alpha)_\ell=\left(\gamma_1-\alpha_1,\cdots,\gamma_\ell-\alpha_\ell+1,\cdots,\gamma_n-\alpha_n\right).
		\end{equation}
		It follows that the entries of \eqref{eq:koopmatrix0} are given by
		\begin{equation}\label{eq:koop_matrix1_hardy}
			\left\langle L_Fe_{k},e_{j} \right\rangle = 
			\begin{cases}\sum_{\ell=1}^n \alpha_\ell(k)  \, a_{\ell,(\alpha(j)-\alpha(k))_\ell}  & \textrm{if } |\alpha(j) |\geq |\alpha(k)| \\
				0 &  \textrm{if } |\alpha(j) |< |\alpha(k)|.
			\end{cases}
		\end{equation}
		
		In the case of other RKHS of analytic functions with a basis of weighted monomials, the entries are still given by the above formulas, but the off-diagonal entries are multiplied by the ratio $\sqrt{\omega_{\alpha(k)}}/\sqrt{\omega_{\alpha(j)}}$.
		
		\begin{remark}\label{rem:linear_vector_field}
			For the linear part of the vector field $F$, where $|\alpha(j)|=1$, $j=1,\cdots,n$, it is clear that $\alpha(j)$ is the canonical basis vector of $\mathbb{C}^n$, i.e. $\alpha_i(j)=\delta_{ij}$, and we have that $a_{\ell,\alpha(j)}=[JF(0)]_{\ell j}$.
			Also, if $|\alpha(j)|=|\alpha(k)|$, we have that $(\alpha(j)-\alpha(k))_\ell=\alpha(r)$ for some $r\leq n$ (i.e. $|\alpha(r)|=1$), with $\alpha_r(j)=\alpha_r(k)+1$, $\alpha_\ell(j)=\alpha_\ell(k)-1$, and $\alpha_i(j)=\alpha_i(k)$ for all $i\notin \{l,r\}$. Then, it follows from \eqref{eq:koop_matrix1_hardy} that

			\begin{equation}\label{eq:total_degree_bloc_hardy}
				\left\langle L_Fe_{k},e_{j} \right\rangle =  
				\begin{cases}\sum_{\ell=1}^{n}\, \alpha_\ell(j) \left[JF(0)\right]_{\ell\ell}  & \textrm{if } j=k \\
					\alpha_{\ell}(k) \left[JF(0)\right]_{\ell r}  & \textrm{if } \alpha(j)=(\alpha_1(k),\cdots, \alpha_{\ell}(k)-1,\cdots, \alpha_{r}(k)+1, \cdots,\alpha_n(k)),\\
					0 &  \textrm{otherwise }.
				\end{cases}
			\end{equation}

		\end{remark}

		\subsubsection{Switched Koopman systems and Lie-algebraic conditions}
		
		In the case of a switched nonlinear system \eqref{eq:switch}, the Koopman operator description yields
		a switched linear infinite-dimensional system (in short, \emph{switched Koopman system}) of the form
		\begin{equation}\label{eq:switch_inf}
			\left\lbrace \dot{f}=L_{F^{(i)}}f, \, f\in  \mathcal{D} \right\rbrace_{i=1}^m
		\end{equation}
		with $\mathcal{D}=\cap_{i=1}^m \mathcal{D}(L_{F^{(i)}})$. Similarly, the Lie algebra $\mathfrak{g}_F$ spanned by $F^{(i)}$ (see \eqref{Lie_algebra_vec}) is replaced by $\mathfrak{g}_\ell=\textrm{span}\left\lbrace L_{F^{(i)}}, i=1,\dots,m\right\rbrace_\textrm{Lie}$, equipped with the Lie bracket
		\begin{equation*}
			\left[ L_{F^{(i)}}, L_{F^{(j)}}\right]=L_{F^{(i)}} L_{F^{(j)}} - L_{F^{(j)}} L_{F^{(i)}} \,.
		\end{equation*}
		In particular, we have the well-known relationship
		\begin{equation}
			\label{eq:liebracket_relation}
			\left[ L_{F^{(i)}}, L_{F^{(j)}}\right] =L_{[F^{(i)},F^{(j)}]}
		\end{equation}
		so that the two algebras $\mathfrak{g}_F$ and $\mathfrak{g}_\ell$ are isomorphic. It follows that Lie-algebraic conditions in $\mathfrak{g}_F$ can be recast into Lie-algebraic criteria in $\mathfrak{g}_\ell$, a framework where we can expect to obtain new results on switched systems that are reminiscent to the linear case. In particular, since the solvability property of $\mathfrak{g}_F$ is equivalent to the solvability property of $\mathfrak{g}_\ell$, we will investigate whether this latter condition implies the existence of a common Lyapunov functional for the switched Koopman system \eqref{eq:switch_inf}.

		\section{Main result}
		\label{sec:mainresult}
		
		This section presents our main result. We first use an illustrative example to show that Lie's theorem \ref{thm:lie} cannot be used for nonlinear vector fields, in contrast to the linear case (see Proposition \ref{prop:clf}). We then relax the algebraic conditions suggested in \cite{liberzon2004lie}  in order to obtain a triangular form in the Koopman matrix representation \eqref{eq:koopmatrix0}, a property which is equivalent to the existence of an invariant flag for the adjoint operator $L_F^*$. We finally prove uniform stability of switched nonlinear systems under these conditions.
		
		\subsection{A first remark on the existence of the common invariant flag}
		
		The following example shows that Lie's theorem does not hold for infinite-dimensional switched Koopman systems.
		
		\begin{example}\label{example_flag}
			Consider the two vector fields
			$$F^{(1)}(x_1,x_2)=(-\alpha x_1,-\alpha x_2)\quad \textrm{and} \quad  F^{(2)}(x_1,x_2)=(-\beta x_1+ \gamma \left( x_1^2-x_2^2\right),-\beta x_2+2 \gamma x_1x_2),$$
			where $\alpha, \beta$ and $\gamma$ are real parameters.
			These two vector fields generate the Lie algebra $\mathfrak{g}=\textrm{span}\left\lbrace F^{(1)},F^{(2)},F^{(3)}\right\rbrace_{Lie}$ with $F^{(3)}(x_1,x_2)=(\alpha \gamma (x_1^2-x_2^2),2\alpha \gamma x_1x_2)$ since $[F^{(1)},F^{(2)}]=F^{(3)}$, $[F^{(1)},F^{(3)}]=\alpha F^{(3)}$ and $[F^{(2)},F^{(3)}]=\beta F^{(3)}$. Moreover, one has $\mathfrak{g}^1=[\mathfrak{g}, \mathfrak{g}]=\textrm{span}\left\lbrace F^{(3)}\right\rbrace_{Lie}$ and $\mathfrak{g}^2=[\mathfrak{g}^1, \mathfrak{g}^1]=0$, which implies that $\mathfrak{g}$ is a solvable Lie algebra. However, the Koopman generators $L_{F^{(1)}}$ and $L_{F^{(2)}}$ associated with the two vector fields do not share a common eigenfunction, and therefore cannot have a common invariant flag. Indeed, the principal eigenfunctions of $L_{F^{(1)}}$ are $\phi_{\tilde{\lambda}_1^{(1)}}(x_1,x_2)=x_1$ and $\phi_{\tilde{\lambda}_2^{(1)}}(x_1,x_2)=x_2$, while those of $L_{F^{(2)}}$ are given by
			$$\phi_{\tilde{\lambda}_1^{(2)}}(x_1,x_2)= \dfrac{ \beta \gamma \left(  \beta x_1- \gamma \left( x_1^2- x_2^2\right) \right)}{( \beta-\gamma x_1)^2+\gamma^2x_2^2} \quad \textrm{and} \quad \phi_{\tilde{\lambda}_2^{(2)}}(x_1,x_2)= \dfrac{ \beta^2 \gamma x_2}{( \beta-\gamma x_1)^2+\gamma^2x_2^2}.$$
		\end{example}
		
		We conclude that Lie's theorem \ref{thm:lie} does not hold setting for the above example, so that we cannot directly extend Proposition \ref{prop:clf} to this case. The two Koopman generators are not simultaneous triangularizable and do not have a common invariant flag (see \cite{katavolos1990simultaneous} for more details about simultaneous triangularization of operators and its connection to the existence of an invariant infinite maximal flag). However, it can be easily seen that the Koopman infinite matrices \eqref{eq:koopmatrix0} related to the vector fields $F^{(1)}$ and $F^{(2)}$ are both lower triangular, and therefore admit a common infinite invariant maximal flag, when they are represented in a RKHS of analytic functions.
		In fact, this implies that the \emph{adjoint} operators $L_{F^{(i)}}^*$ have a common invariant flag generated by the orthonormal basis $e_k$, i.e. $S_k=\textrm{span}\{e_1,\dots,e_k\}$. Indeed, $L_{F^{(i)}}^* e_{j} \in \textrm{span}\{e_1,\dots,e_j\}$ since $\left\langle e_{k}, L_{F^{(i)}}^* e_{j} \right\rangle=0$ for all $k>j$. For this reason, we will depart from the solvability condition on vector fields (i.e. on Koopman generators), and we will impose instead the existence of a common invariant flag for the adjoints of Koopman generators.
		
		\begin{assumption}\label{assumption:lower_triangular}
				The Koopman matrix \eqref{eq:koopmatrix0}
				is lower triangular, i.e.	$\left\langle L_{F^{(i)}} e_{k}, e_{j} \right\rangle=0$ for all $k<j$ and $i$.
		\end{assumption}
		In Lemma \ref{lemm:triangular_form} below, we will see that for RKHS where  the orthonormal basis is given by the weighted monomials $e_\alpha(z)=\sqrt{\omega_\alpha}z^\alpha$, Assumption \ref{assumption:lower_triangular} is equivalent to the triangular structure of the Jacobian matrix of the vector field $F$, which leads to a condition that is weaker than the solvability property of the vector fields.
		
		\subsection{ A common Lyapunov function for switched nonlinear systems}

		We now aim to show that, for some positive sequence $(\epsilon_k)_{k=1}^\infty$, the series
		\begin{equation}\label{eq:lyapunov1}
			\mathcal{V}(f)=\sum_{k=1}^{\infty}\epsilon_k  \left| \left\langle f, e_k\right\rangle \right|^2
		\end{equation}
		is a Lyapunov functional for the  switched Koopman system \eqref{eq:switch_inf}. Before starting our main result, we need the following lemma.

		\begin{lemma}\label{lemmepreuve3}
			Let $\dot z =F(z)$ be a nonlinear system on $D$ and let $ \dot f =L_{F}f $ be its corresponding Koopman  system on $\mathcal{D}\left(L_F\right)\subset \mathcal{H}^2(D)$ such that Assumptions \ref{assumption:general}-\ref{assumption:lower_triangular} are satisfied.
			
			If the series  
			\begin{equation}\label{eq:serie1}
					\mathcal{V}(k_z-k_0)=\sum_{k=1}^{\infty}\epsilon_k  \left| \left\langle k_z, e_k\right\rangle \right|^2
				\end{equation} and
				\begin{equation}
					\label{eq:serie2}
					\sum_{k=1}^{\infty} \epsilon_k \dfrac{d}{dt}  \left| \left\langle U_t^* (k_z-k_0), e_k\right\rangle \right|^2
			\end{equation}
			are absolutely and uniformly convergent in $S \subseteq D\subseteq\mathbb{B}^n$ for all $t>0$, then, for all double sequences of positive real numbers $(b_{jk})_{j\geq 1,k\geq 1}$ such that
			\begin{equation}\label{eq:pondera}
				\sum_{k=1}^\infty b_{jk}\leq 1,
			\end{equation}
			one has, for all $z\in S  \subseteq D\subseteq\mathbb{B}^n$,
			{\small
				\begin{equation}
					\label{eq:time_der_Lyap}
					\begin{split} \dfrac{d}{dt}\mathcal{V}\left(U_t^*(k_z-k_0)\right)\leq & 2\sum_{j=1}^{\infty} b_{jj} \epsilon_j   \left| c_{j}\right|^2 \Re\left(\lambda_{j}\right)\\ 
						& +2\sum_{j=2}^{\infty} \sum_{k=1}^{j-1} \left(b_{jk} \epsilon_j \left| c_{j}\right|^2 \Re\left(\lambda_{j}\right)+ b_{kj}\epsilon_k \left| c_{k}\right|^2 \Re\left(\lambda_{k}\right)+ \epsilon_k \Re \left( c_{j}\bar c_{k}  \left\langle  e_j,L_{F}e_k\right\rangle \right)\right),
					\end{split}
			\end{equation}}
			with $c_j=\left\langle U_t^*(k_z-k_0), e_j\right\rangle $ and $\lambda_{j}=\left\langle L_F e_j, e_j\right\rangle $. 
		\end{lemma}
		\begin{proof} Suppose that $z\in S  \subseteq  D\subseteq \mathbb{B}^n$ is such that
			the series \eqref{eq:serie1} and 
			\eqref{eq:serie2} are absolutely and uniformly convergent. Then, by using Lemma \ref{lemmepreuve1} (3), we obtain
			\begin{equation*}
				\begin{split}
					\dfrac{d}{dt} \epsilon_k  \left| \left\langle U_t^*(k_z-k_0), e_k\right\rangle \right|^2 
					&= 2\epsilon_k \Re \left( \left\langle \dfrac{d}{dt} U_t^*(k_z-k_0), e_k\right\rangle \overline{\left\langle  U_t^*(k_z-k_0), e_k \right\rangle}  \right)\\
					&= 2\epsilon_k \Re \left( \left\langle L_{F}^* U_t^*(k_z-k_0), e_k\right\rangle \overline{\left\langle  U_t^*(k_z-k_0), e_k \right\rangle}  \right)\\
					&= 2\epsilon_k \sum_{j=1}^{\infty} \Re \left(c_{j} \bar c_{k}  \left\langle  e_j,L_{F}e_k\right\rangle  \right)
				\end{split}
			\end{equation*}
			where we used the decomposition $U_t^* (k_z-k_0)= \sum_{j=1}^{\infty} c_{j} e_j$.
			Since \eqref{eq:serie2} is  absolutely and  uniformly convergent, term by term derivation yields
			\begin{equation*}
				\begin{split}
					\dfrac{d}{dt} \mathcal{V}\left( U_t^*(k_z-k_0) \right)&=\sum_{k=1}^{\infty}  \dfrac{d}{dt}  \epsilon_k  \left| \left\langle U_t^*(k_z-k_0), e_k\right\rangle \right|^2\\
					&= 2 \sum_{k=1}^{\infty} \sum_{j=1}^{\infty} \epsilon_k \Re \left(c_{j} \bar c_{k}  \left\langle  e_j,L_{F}e_k\right\rangle  \right)\\
					&= 2 \sum_{j=1}^{\infty} \epsilon_j  \left| c_{j}\right|^2  \Re \left( \lambda_{j} \right)+ 2 \sum_{j=2}^{\infty} \sum_{k=1}^{j-1} \epsilon_k \Re \left(c_{j} \bar c_{k}  \left\langle  e_j,L_{F}e_k\right\rangle  \right)
				\end{split}
			\end{equation*}
			where we used the triangular form of $\bar L_{F}$ (Assumption \ref{assumption:lower_triangular}) and $\lambda_j=\left\langle L_F e_j, e_j\right\rangle$.

			Using \eqref{eq:pondera}, we have
			{\small
				\begin{equation*}
					\begin{split}
						\dfrac{d}{dt} \mathcal{V}\left( U_t^*(k_z-k_0) \right) &\leq 2\sum_{j=1}^{\infty} \left( \sum_{k=1}^{j}b_{jk}+\sum_{k=j+1}^{\infty}b_{jk}\right) \epsilon_j \left| c_{j}\right|^2 \Re\left(\lambda_{j}\right)+2\sum_{j=2}^{\infty}\sum_{k=1}^{j-1} \epsilon_k \Re\left( c_{j}\bar c_{k}  \left\langle  e_j,L_{F}e_k\right\rangle \right)\\
						&= 2\sum_{j=1}^{\infty} b_{jj} \epsilon_j \left| c_{j}\right|^2 \Re\left(\lambda_{j}\right)+2\sum_{j=2}^{\infty} \sum_{k=1}^{j-1}b_{jk} \epsilon_j \left| c_{j}\right|^2 \Re\left(\lambda_{j}\right)\\
						&\quad +2 \sum_{j=1}^{\infty} \sum_{k=j+1}^{\infty}b_{jk}\epsilon_j \left| c_{j}\right|^2 \Re\left(\lambda_{j}\right)+2\sum_{j=2}^{\infty}\sum_{k=1}^{j-1} \epsilon_k \Re \left( c_{j}\bar c_{k}  \left\langle  e_j,L_{F}e_k\right\rangle \right)\\
						&= 2\sum_{j=1}^{\infty} b_{jj} \epsilon_j   \left| c_{j}\right|^2 \Re\left(\lambda_{j}\right) \\
						&\quad+2\sum_{j=2}^{\infty} \sum_{k=1}^{j-1} \left(b_{jk} \epsilon_j \left| c_{j}\right|^2 \Re\left(\lambda_{j}\right)+ b_{kj}\epsilon_k \left| c_{k}\right|^2 \Re\left(\lambda_{k}\right)+ \epsilon_k \Re \left( c_{j}\bar c_{k}  \left\langle  e_j,L_{F}e_k\right\rangle \right)\right).
					\end{split}
			\end{equation*}}
			\hfill{}
		\end{proof}
		
		It is important to note that, when $\Re\{\lambda_{j}\}<0$ for all $j$, the time derivative \eqref{eq:time_der_Lyap} of the Lyapunov functional is negative if (negative) terms related to the diagonal entries $\left\langle L_F e_j, e_j\right\rangle$ and $\left\langle L_F e_k, e_k \right\rangle $ compensate (possibly positive) cross-terms related to $\left\langle  e_j,L_{F}e_k\right\rangle$. We note that a term associated with a diagonal entry will be used to compensate an infinity of cross-terms (associated with entries in the corresponding row and column of the Koopman matrix), and the values $b_{jk}$ play the role of weights in the compensation process.
		
		We are now in position to state our main result.
		\begin{theorem}\label{thm:principalresult}
			Let 
			\begin{equation}\label{eq:switchonpoly}
				\left\lbrace \dot z =F^{(i)}(z) \right\rbrace_{i=1}^m
			\end{equation}
			be a switched nonlinear system on $D$ 
			and suppose that
			\begin{itemize}
				\item Assumptions \ref{assump1}-\ref{assumption:lower_triangular} hold,
				\item all subsystems of \eqref{eq:switchonpoly} have a common hyperbolic equilibrium $z_e=0$ that is globally asymptotically stable on $D$,
					\item $\Re\{\left\langle L_{F^{(i)}} e_j, e_j\right\rangle\}<0$ for all $j\in \mathbb{N}$ and $i\in \{1,\cdots,m\}$,
					\item there exists $\rho\in ]0, 1]$ such that  a set $S\subseteq D\subseteq\mathbb{B}^n\left( \rho \right)$  is forward invariant with respect to the flows $\varphi^{(i)}_t$ of $F^{(i)}$.
				\end{itemize}
				Consider a double sequence of positive real numbers $\left( b^{(i)}_{jk}\right)_{j\geq 1, k\geq 1}$, with $i=1,\dots,m$, such that $b^{(i)}_{jk}b^{(i)}_{kj}>0$ if $\langle L_{F^{(i)}} e_k,  e_j \rangle \neq 0$  and such that $\sum_{k=1}^\infty b^{(i)}_{jk}\leq 1$, and define the double sequence
				\begin{equation}\label{eq:double_sequence}
					\left( Q^{(i)}_{jk} \stackrel{\text{def}}{=} \begin{cases} \dfrac{\left|\left\langle L_{F^{(i)}} e_k,  e_j \right\rangle \right|^2  }{4\left|\Re\left(\left\langle   L_{F^{(i)}} e_j, e_j \right\rangle\right)\right|\left|\Re\left(\left\langle L_{F^{(i)}} e_k,  e_k \right\rangle\right)\right|}\dfrac{1}{b^{(i)}_{jk}b^{(i)}_{kj}} & \textrm{if } \left\langle L_{F^{(i)}} e_k,  e_j \right\rangle \neq 0 \\
						0 & \textrm{otherwise} \end{cases} \right)_{j\geq 2, 1\leq k\leq j-1}.
				\end{equation}
				If the series \eqref{eq:serie1} and \eqref{eq:serie2}
					are absolutely and uniformly convergent in $S \subseteq D\subseteq\mathbb{B}^n$ for all $t>0$, with
				
				\begin{equation}
					\label{eq:defineepsilon}
					\epsilon_j = \max_{\substack{i=1,\cdots,m \\ k=1,\dots, j-1}} \epsilon_k \, Q^{(i)}_{jk},
				\end{equation}
				
				then the switched system \eqref{eq:switchonpoly} is GUAS on $S\subseteq D$. Moreover the series 
				$$V(z)=\sum_{k=1}^\infty \epsilon_k \left\vert z^{\alpha (k)}\right\vert^2$$
				is a common global Lyapunov function on $ S\subseteq  D$).
			\end{theorem}
			
			\begin{proof}
				Consider the switched system
				\begin{equation}\label{eq:switchonpoly_new_coord}
					\left\lbrace \dot{z} =F^{(i)}(z) \right\rbrace_{i=1}^m
				\end{equation}
				defined on  $D$ and generating the flows $\varphi_t^{(i)}$ invariant on $S\subseteq D\subseteq \mathbb{B}^n$. By Assumption \ref{assumption:lower_triangular}, the orthonormal basis $\{e_k\}_{k=0}^\infty$ generates a common infinite invariant maximal flag for $\bar L_{F^{(i)}}$. We first show that the candidate Lyapunov functional $\mathcal{V}(f)=\sum_{k=1}^{\infty}\epsilon_k  \left| \left\langle f, e_k\right\rangle \right|^2$
				satisfies
				\begin{equation*}
					\dfrac{d}{dt}\mathcal{V}\left( (U^{(i)}_t)^*\left(k_{z}-k_0\right)\right)<0
				\end{equation*}
				for all $i=1,\cdots, m$ and all $z\in  S\subseteq  D$, where $U^{(i)}_t$ denotes the Koopman semigroup associated with the subsystem $\dot{z}=F^{(i)}(z)$. 
				Since, by assumption, the series \eqref{eq:serie1} and \eqref{eq:serie2} are absolutely convergent on $S$, it follows from Lemma \ref{lemmepreuve3} that, for all $z \in   S\subseteq D$,
				\begin{equation*}
					\begin{split}
						\dfrac{d}{dt}\mathcal{V}\left(( U^{(i)}_t)^*\left(k_{z}-k_0\right)\right)\leq & 2\sum_{j=1}^{\infty} b^{(i)}_{jj} \epsilon_j   \left| c^{(i)}_{j}\right|^2 \Re\left(\lambda^{(i)}_{j}\right)\\ &+2\sum_{j=2}^{\infty} \sum_{k=1}^{j-1} b^{(i)}_{jk} \epsilon_j \left| c^{(i)}_{j}\right|^2 \Re\left(\lambda^{(i)}_{j}\right)+ b^{(i)}_{kj}\epsilon_k \left| c^{(i)}_{k}\right|^2 \Re\left(\lambda^{(i)}_{k}\right)+ \epsilon_k \Re \left( c^{(i)}_{j}\bar c^{(i)}_{k}  \left\langle e_j, L_{F^{(i)}} e_k\right\rangle \right)
					\end{split}
				\end{equation*}
				where $c^{(i)}_{j}=\left\langle (U^{(i)}_t)^*\left(k_{z}-k_0\right), e_j\right\rangle $ and $\lambda^{(i)}_{j}=\left\langle L_{F^{(i)}} e_j, e_j\right\rangle$.
				Since $\Re\{\lambda^{(i)}_{j}\}<0$, one has to find a sequence of positive numbers $\left(\epsilon_j\right)_{j\geq 1}$ such that
				\begin{equation*}
					b^{(i)}_{jk} \epsilon_j \left| c^{(i)}_{j}\right|^2 \left|\Re\left(\lambda^{(i)}_{j}\right)\right|+ b^{(i)}_{kj}\epsilon_k \left| c^{(i)}_{k}\right|^2 \left|\Re\left(\lambda^{(i)}_{k}\right)\right| \geq \epsilon_k \left| \Re \left( c^{(i)}_{j}\bar c^{(i)}_{k}  \left\langle  e_j, L_{F^{(i)}} e_k\right\rangle \right) \right|
				\end{equation*}
				for all $i=1,\cdots, m$ and for all $j,k$ with $j>k$ such that 
				\begin{equation}\label{eq:compensate_nonzero_term}
					\left\langle  e_j,L_{F^{(i)}}e_k\right\rangle \neq 0 .
				\end{equation}
				By using the inequality 
				\begin{equation*}\label{eq:inequality}
					\left| \Re \left( c^{(i)}_{j}\bar c^{(i)}_{k}  \left\langle  e_j, L_{F^{(i)}} e_k\right\rangle \right) \right|\leq \left| c^{(i)}_{j}\right| \left| c^{(i)}_{k}\right| \left| \left\langle e_j,L_{F^{(i)}}e_k\right\rangle \right|,
				\end{equation*}
				one has to satisfy
				\begin{equation*}
					b^{(i)}_{jk}\epsilon_j \left| c^{(i)}_{j}\right|^2 \left| \Re\left(\lambda^{(i)}_{j}\right)\right| + b^{(i)}_{kj}\epsilon_k \left| c^{(i)}_k\right|^2 \left| \Re\left(\lambda^{(i)}_{k}\right)\right| \geq \epsilon_k \left| c^{(i)}_{j}\right| \left| c^{(i)}_{k}\right| \left| \left\langle L_{F^{(i)}} e_k, e_j\right\rangle \right|
				\end{equation*}
				or equivalently
				\begin{equation}
					\label{eq:cond_epsilon}
					\epsilon_j \geq \epsilon_k \left(-\dfrac{ b^{(i)}_{kj}}{b^{(i)}_{jk}} \dfrac{\left| \Re\left(\lambda^{(i)}_{k}\right)\right| }{\left| \Re\left(\lambda^{(i)}_{j}\right)\right|} \left| \dfrac{  c^{(i)}_{k}}{  c^{(i)}_{j}}\right|^2 +    \dfrac{\left| \left\langle L_{F^{(i)}} e_k, e_j\right\rangle \right|}{b^{(i)}_{jk}\left| \Re\left(\lambda^{(i)}_{j}\right)\right| } \left| \dfrac{  c^{(i)}_{k}}{  c^{(i)}_{j}}\right|\right) \stackrel{\text{def}}{=} \epsilon_k \,  h \left(\left| \dfrac{c^{(i)}_{k}}{c^{(i)}_{j}}\right|\right) .
				\end{equation}
				It is easy to see that the real quadratic function $h$ has the maximal value
				\begin{equation*}
					Q^{(i)}_{jk}
					=\dfrac{\left|\left\langle L_{F^{(i)}} e_k,  e_j \right\rangle\right|^2  }{4\left|\Re\left(\lambda^{(i)}_{j}\right)\right| \left|\Re\left(\lambda^{(i)}_{k}\right)\right|}\dfrac{1}{b^{(i)}_{jk}b^{(i)}_{kj}}
				\end{equation*}
				so that \eqref{eq:cond_epsilon} is satisfied if we choose iteratively $\epsilon_j$ according to \eqref{eq:defineepsilon}. Finally, it follows that we have
				\begin{eqnarray*}
					\dfrac{d}{dt}\mathcal{V} ( \left(U^{(i)}_t)^*\left(k_{z}-k_0\right)\right)&< &2\sum_{j=1}^{\infty} b^{(i)}_{jj} \epsilon_j   \left| c^{(i)}_{j}\right|^2 \Re\left(\lambda^{(i)}_{j}\right) \\
					&< &-\min_{j}\left\lbrace b^{(i)}_{jj}    \left|\Re\left(\lambda^{(i)}_{j}\right)\right|\right\rbrace \sum_{j=1}^{\infty} \epsilon_j \left| c^{(i)}_{j}\right|^2 \\
					&=&-\min_{j}\left\lbrace b^{(i)}_{jj}    \left|\Re\left(\lambda^{(i)}_{j}\right)\right|\right\rbrace  \mathcal{V}\left((U^{(i)}_t)^*\left(k_{z}-k_0\right)\right) .
				\end{eqnarray*}
				With the evaluation functional $k_{z}$, we can define
				\[V:   S\subseteq  D\rightarrow \mathbb{R}^+,\quad V(z)=\mathcal{V}\left(k_{z}-k_0\right)\]
				and, using Lemma \ref{lemmepreuve1}, we verify that
				\begin{equation*}
					\begin{split}
						V\left( \varphi_t^{(i)} (z)\right) = \mathcal{V}\left(k_{ \varphi_t^{(i)} (z)}-k_0\right) = \mathcal{V}\left(k_{\varphi_t^{(i)} (z)}-k_{\varphi_t^{(i)}(0)}\right) & =\mathcal{V}\left( ( U^{(i)}_t)^* \left(k_{z}-k_0\right)\right) \\
						& < \mathcal{V}\left(  k_{z}-k_0\right) = V(z).
					\end{split}
				\end{equation*}
				%
				Therefore, we have the CLF 
				\begin{equation}\label{eq:lyapunovnonlinear}
					V(z)=\sum_{k=1}^\infty \epsilon_k \left| \left\langle k_{z}-k_0,e_k\right\rangle \right|^2=\sum_{k=1}^\infty \epsilon_k \left| \left\langle k_{z},e_k\right\rangle \right|^2=\sum_{k=1}^\infty \epsilon_k \left\vert z^{\alpha (k)}\right\vert^2
				\end{equation}
				for the switched nonlinear system \eqref{eq:switchonpoly}. Finally, since $ S\subseteq D$ is forward invariant with respect to $\varphi^{(i)}_t$, the switched system \eqref{eq:switchonpoly} is GUAS on $S\subseteq D$.
			\end{proof}

			Note that, if the assumptions of Theorem \ref{thm:principalresult} are satisfied but the set $ S\subseteq D$  is not forward invariant with respect to the flow generated by the subsystems, then the switched system is GUAS in the largest sublevel set of the Lyapunov function that is contained in $D$.

			The condition on the boundedness of the double sequence \eqref{eq:double_sequence} could be interpreted as the dominance of the negative diagonal entries of the matrix $\bar{L}_{F^{(i)}}$ with respect to the other entries.  Moreover, the number of nonzero cross-terms \eqref{eq:compensate_nonzero_term} to be compensated affects the way we define the sequence of weights $b^{(i)}_{jk}$ and therefore the sequence $\epsilon_j$ in \eqref{eq:defineepsilon}.  
				If the double sequence \eqref{eq:double_sequence} has an upper bound $Q<1$, one can set $\epsilon_k=Q$ for all $k$. However, such case rarely appears. Instead, if $Q>1$, one might have $\epsilon_k =\mathcal{O}(Q^k)$ and the series \eqref{eq:serie1} and \eqref{eq:serie2} might not converge. In the next section, we will focus on specific cases where we can find a proper choice of sequence $b^{(i)}_{jk}$ such that the series converge.
				
				\subsection{Corollaries of the main result in the Hardy space context}
				Now, we focus on the case of the Hardy space on the polydisk (or the hypercube), for which we derive two corollaries of Theorem \ref{thm:principalresult}.
				First, we recall that Assumption \ref{assumption:general} (i) (i.e. analytic vector field) implies that Assumption \ref{assumption:general} (ii) is satisfied (see Remark \ref{rem:assmp2} above.) Next, we prove that Assumption \ref{assumption:lower_triangular} also holds if the Jacobian matrix is triangular.
				\begin{lemma}\label{lemm:triangular_form}
					Let $F$ be an analytic vector field on $\mathbb{D}^n$ such that the Jacobian matrix $JF(0)$ is upper triangular. Then the Koopman matrix \eqref{eq:koopmatrix0} is lower triangular, i.e.
					$\left\langle L_F e_{k}, e_{j} \right\rangle=0$ for all $k<j$.
					Moreover, the adjoint $L_F^*$ of the Koopman generator admits an infinite invariant maximal flag generated by the monomials $e_k$, i.e. $S_k=\textrm{span}\{e_1,\dots,e_k\}$.
				\end{lemma}
				\begin{proof}
					It follows from \eqref{eq:koop_matrix1_hardy} that $\left\langle L_F e_{k}, e_{j} \right\rangle=0$ if $|\alpha(k)|>|\alpha(j)|$ (i.e. the Koopman matrix \eqref{eq:koopmatrix0} is always lower triangular by matrix blocks related to monomials of the same total degree). In the case $|\alpha(k)|=|\alpha(j)|$ with $k>j$, the lexicographic order implies that one can have $\alpha(j)=(\alpha_1(k),\cdots, \alpha_{\ell}(k)-1,\cdots, \alpha_{r}(k)+1, \cdots,\alpha_n(k))$ only with $r<l$. Since $[JF(0)]_{\ell r}=0$ for all $l>r$, it follows from \eqref{eq:total_degree_bloc_hardy} that $\left\langle L_F e_{k}, e_{j} \right\rangle=0$ when $k>j$.
					Finally, it is clear that $L_F^* e_{j} \in \textrm{span}\{e_1,\dots,e_j\}$ since $\left\langle e_{k}, L_F^* e_{j} \right\rangle=0$ for all $k>j$.
				\end{proof}
				Note that the previous lemma remains valid for any RKHS of analytic functions with an orthonormal basis of weighted monomials.
				
				\begin{remark}
					\label{rem:eigen}
					When the Jacobian matrix is upper triangular, it is well-known that $[JF(0)]_{jj}=\tilde{\lambda}_j$. In this case, it follows from \eqref{eq:total_degree_bloc_hardy} that the diagonal entries of the (lower triangular) Koopman matrix are given by
					
					\begin{equation}
						\label{eq:diagonal_hardy}
						\lambda_j = \langle L_F e_j, e_j\rangle= \sum_{\ell=1}^n \alpha_\ell(j) \tilde{\lambda}_\ell.
					\end{equation}
					
					Note that these values are the Koopman eigenvalues in the case of non-resonant\footnote{The eigenvalues $\tilde{\lambda}_j$ are non-resonant if $\sum_{j=1}^n \alpha_j \tilde{\lambda}_j = 0$ with $\alpha \in \mathbb{Z}^n$ implies that $\alpha=0$.} eigenvalues $\tilde{\lambda}_j$.
				\end{remark}

				It follows from Lemma \ref{lemm:triangular_form} that we can guarantee the existence of a common invariant maximal flag for the family of adjoints of Koopman generators.
				\begin{corollary}
					\label{corol:jacobian_solvable} Let $\left\lbrace F^{(i)}\right\rbrace _{i=1}^m$ be a switched nonlinear system on $\mathbb{D}^n $ (or $\mathbb{X}^n $) and suppose that the Lie algebra of matrices $\textrm{span}\left\lbrace JF^{(i)}(0)\right\rbrace _{Lie}$ is solvable. Then there exists a change of variables $z\mapsto \widehat{z}= P^{-1}z $ from $\mathbb{C}^n$ (or $\mathbb{R}^n$) to $\mathbb{C}^n$ such that the adjoint operators $L_{\widehat{F}^{(i)}}^*$ of the Koopman generators (with $\widehat{F}^{(i)}(\widehat{z})=P^{-1}F^{(i)}(P\widehat{z})$) admit a common infinite invariant maximal flag. Moreover, $\left\lbrace \widehat{F}^{(i)}\right\rbrace _{i=1}^m$ is a switched nonlinear system on $S=P^{-1}(\mathbb{D}^n) \subseteq \mathbb{D}^n\left(\| P^{-1}\|_\infty \right)$.
				\end{corollary}
				\begin{proof}
					Since $\textrm{span}\left\lbrace JF^{(i)}(0)\right\rbrace _{Lie}$ is solvable, Lie's theorem \ref{thm:lie} implies that the matrices $JF^{(i)}(0)$ are simultaneously triangularizable, i.e. there exists a (complex) matrix $P$ such that $JF^{(i)}(0)=PT^{(i)}P^{-1}$ for all $i$, where $T^{(i)}$ is upper triangular. 
					Let set $F^{(i)}(z)=JF^{(i)}(0)z+\tilde F^{(i)}(z)$ to separate the linear and the nonlinear parts of the dynamics. In the new coordinates $\widehat{z}=P^{-1}z$, we obtain the dynamics $\widehat{F}^{(i)}(\widehat{z})=P^{-1}JF^{(i)}(0)P \widehat{z}+P^{-1}\tilde{F}_i(P\widehat{z})=T^{(i)}\widehat{z}+ \widehat{\tilde{F}}_i(\widehat{z})$.
					It follows from Lemma \ref{lemm:triangular_form} that monomials $\widehat{e}_k$, with $\widehat{e}_{k}(\widehat{z})=(\widehat{z})^{\alpha(k)}$, generate a common invariant maximal flag for $L^*_{\widehat{F}^{(i)}}$. In addition, for all $z\in \mathbb{D}^n$ and all $j=1,\dots, n$, we have
					\begin{equation*}
						\left|\widehat{z}_j \right|\leq \left\|\widehat{z} \right\|_{\infty} = \left\|P^{-1} z \right\|_{\infty} \leq \left\|P^{-1}\right\|_{\infty} \, \left\|z \right\|_{\infty} < \| P^{-1}\|_\infty .
					\end{equation*}
				\end{proof}
				It is clear that the change of coordinates $z\mapsto \widehat{z}= P^{-1}z$ is defined up to a multiplicative constant. Without loss of generality, we will consider in the sequel that $\| P^{-1}\|_\infty=1$, so that $\left\lbrace \widehat{F}^{(i)}\right\rbrace _{i=1}^m$ is a switched nonlinear system on $S=P^{-1}(\mathbb{D}^n) \subseteq \mathbb{D}^n$.
				
				\begin{remark} \label{rem:solvability}
					Even in the case of a system defined on the real hypercube, the matrix $P$ can be complex, and it might not be possible to obtain the simultaneous triangularization property of the Jacobian matrices with only real coordinates. It follows that the common invariant maximal flag for the generators $L^*_{\widehat{F}^{(i)}}$ is obtained in general on $\mathbb{H}^2(\mathbb{D}^n)$, and not on $\mathbb{H}^2(\mathbb{X}^n)$.
				\end{remark}
				
				Instead of a nilpotency or solvability condition on the vector fields $F^{(i)}$, we only require a milder solvability condition on the Jacobian matrices $JF^{(i)}(x_e)$ to guarantee the triangular form of the Koopman matrix \eqref{eq:koopmatrix0}. It is noticeable that this \emph{local} condition is much less restrictive than the  \emph{global} solvability condition mentioned in the original open problem \cite{liberzon2004lie}. Also, it was shown in \cite{angeli2000note} that the triangular form of the vector fields (and therefore of the Jacobian matrices) is not sufficient to guarantee the GUAS property of a switched nonlinear system on $\mathbb{R}^n$. In the next corollaries, however, we use the solvability condition on the Jacobian matrices to prove the GUAS property in a bounded invariant region of the state space. This result is consistent with the local stability result derived in \cite{Liberzon2013SwitchedS}.
				
				As a final step, we prove the following lemma related to the convergence of the series \eqref{eq:serie1} and \eqref{eq:serie2}, which is required in Lemma \ref{lemmepreuve3} and Theorem \ref{thm:principalresult}.
				\begin{lemma}\label{lemmepreuve2}
					Let $\dot z=F(z)$ be a vector field on $\mathbb{D}^n$   which generates a flow $\varphi_t$. Suppose that
					there exist a sequence of positive numbers $\left(\epsilon_k\right)_{k\geq 1}$   and $\rho\in ]0, 1]$ such that a set $S \subseteq \mathbb{D}^n(\rho)$  is forward invariant with respect to $\varphi_t$ and such that the series
					\begin{equation}
						\label{eq:conditionconvergence_bis}
						\sum_{k= 1}^\infty |\alpha(k)|\epsilon_k\rho^{2|\alpha(k)|}
					\end{equation}
					is convergent. 
					Then, the series \begin{equation}\label{eq:lemmepreuve2_serie1}
						\mathcal{V}(k_z-k_0)=\sum_{k=1}^{\infty}\epsilon_k  \left| \left\langle k_z, e_k\right\rangle \right|^2
					\end{equation} and
					\begin{equation}
						\label{eq:lemmepreuve2_serie2}
						\sum_{k=1}^{\infty} \epsilon_k \dfrac{d}{dt}  \left| \left\langle U_t^* (k_z-k_0), e_k\right\rangle \right|^2
					\end{equation}
					are absolutely and uniformly convergent  on $S$ for all $t>0$.  
				\end{lemma}
				
				\begin{proof}
					For the first series, we have
					\begin{equation*}
						\epsilon_k  \left| \left\langle k_z-k_0, e_k\right\rangle \right|^2=\epsilon_k  \left| \left\langle k_z, e_k\right\rangle \right|^2
						=\epsilon_k  \left| z^{\alpha(k)} \right|^{2}<\epsilon_k \rho^{2|\alpha(k)|} \leq |\alpha(k)| \epsilon_k \rho^{2|\alpha(k)|}
					\end{equation*}
					for all $z\in S$ and all $k\geq 1$. For the second series, we have
					{\small
						\begin{eqnarray*}
							\epsilon_k \left|\dfrac{d}{dt}   \left| \left\langle U_t^*(k_z-k_0), e_k\right\rangle \right|^2\right|
							&=&\epsilon_k \left|\dfrac{d}{dt}   \left|(\varphi_t( z))^{\alpha(k)}-(\varphi_t( 0))^{\alpha(k)} \right|^2\right|\\
							&=&2 \epsilon_k \left|  \Re \left(\dfrac{d}{dt} \left( \varphi_t( z) \right)^{\alpha(k)} \cdot \left( \overline{\varphi_t( z)} \right)^{\alpha(k)}\right) \right|\\
							&\leq &2 \epsilon_k  \left|\dfrac{d}{dt} \left( \varphi_t( z) \right)^{\alpha(k)}\right| \cdot  \left| \left( \overline{\varphi_t( z)} \right)^{\alpha(k)} \right|\\
							&< &2 \epsilon_k  \rho^{|\alpha(k)|} \left|\dfrac{d}{dt} \left( \varphi_t( z) \right)^{\alpha(k)}\right| \\
							&\leq &2 \epsilon_k  \rho^{|\alpha(k)|} \sum_{s=1}^n \alpha_s(k) \left| F^{(s)} \left( \varphi_t( z) \right)\right|  \left|\left( \varphi_t( z) \right)_s\right|^{\alpha_s(k)-1} \prod_{\ell=1,l\neq s}^n \left|\left( \varphi_t( z) \right)_\ell\right|^{\alpha_\ell(k)} \\ 
							&< &2 \epsilon_k  \rho^{2|\alpha(k)|-1} \sum_{s=1}^n \alpha_s(k) \left| F^{(s)} \left( \varphi_t( z) \right)\right|
					\end{eqnarray*}}
					for all $z\in S$, $t>0$ and $k\geq 1$.
					By using the invariance of $S$  with respect to $\varphi_t$ and the maximum modulus principle for bounded domains \ref{maxim_modul} with the holomorphic function  $F^{(s)} $, we can denote 
					$$M=\max_{z\in \partial S, s=1,\cdots,n}|  F^{(s)} \left( \varphi_t( z)\right)|$$
					and we obtain
					\begin{equation*}
						\epsilon_k \left|\dfrac{d}{dt}   \left| \left\langle U_t^*(k_z-k_0), e_k\right\rangle \right|^2\right| < 2 M|\alpha(k)| \epsilon_k  \rho^{2|\alpha(k)|-1}.
					\end{equation*}
					Finally, absolute and uniform convergence of both series follow from the Weierstrass test \eqref{wierstrass}.
				\end{proof}
				
				\subsubsection{Global uniform asymptotic stability for polynomial vector fields}
				
				For polynomial vector fields of the form $F^{(i)}_{\ell}(z)=\sum_{k=1}^r a^{(i)}_{\ell,k} z^{\alpha(k)}$, we denote by $K^{(i)}$ the number of nonzero terms (without counting the monomial $z_{\ell}$ in $F^{(i)}_{\ell}$), i.e.
				\begin{equation}
					\label{eq:nb_terms}
					K^{(i)}=\sum_{\ell=1}^m \# \left \{ k \neq \ell:a^{(i)}_{\ell,k} \neq 0 \right\}
				\end{equation}
				where $\#$ is the cardinal of a set. In this case, we have the following result.
				
				\begin{corollary}\label{coroll:poly}
					Let
					\begin{equation}\label{eq:switchonpoly3}
						\left\lbrace \dot z =F^{(i)}(z) \right\rbrace_{i=1}^m
					\end{equation}
					be a switched nonlinear system on $\mathbb{D}^n$ (or on $\mathbb{X}^n$), where  $F^{(i)}$ are polynomial vector fields. Assume that
					\begin{itemize}
						\item all subsystems of \eqref{eq:switchonpoly3} have a common hyperbolic equilibrium $z_e=0$ that is globally asymptotically stable on $\mathbb{D}^n$ (or on $\mathbb{X}^n$),
						\item the Lie algebra $\textrm{span}\left\lbrace JF^{(i)}(0)\right\rbrace_{Lie} $ is solvable (and therefore there exists a matrix $P$ such that $PJF^{(i)}(0)P^{-1}$ are upper triangular),
						\item the polydisk $\mathbb{D}^n$ (or the hypercube $\mathbb{X}^n$) is forward invariant with respect to the flows $\varphi_t^{(i)}$ generated by $F^{(i)}$.
					\end{itemize}
					If 
					\begin{equation}
						\label{eq:maxim1}
						Q=\max_{i=1,\cdots,m} \limsup_{j \in \mathbb{N}} \max_{k=1,\dots,j-1} \dfrac{(\widehat{K}^{(i)})^2 \left| \left\langle  L_{\widehat{F}^{(i)}}\widehat{e}_k,\widehat{e}_j\right\rangle \right|^2}{\left|\Re\left(\left\langle L_{\widehat{F}^{(i)}}\widehat{e}_j, \widehat{e}_j \right\rangle\right)\right|\left|\Re\left(\left\langle L_{\widehat{F}^{(i)}} \widehat{e}_k, \widehat{e}_k \right\rangle\right)\right|}<1,
					\end{equation}
					where $\widehat{K}^{(i)}$ is the number of nonzero terms of $\widehat{F}^{(i)}(\widehat{z})=P^{-1}F^{(i)}(P\widehat{z})$ (see \eqref{eq:nb_terms}) and where $\widehat{e}_j(\widehat{z})=\widehat{z}^{\alpha(j)}$ are the monomials in the new coordinates $\widehat{z}=P^{-1}z$, then \eqref{eq:switchonpoly3} is GUAS on $\mathbb{D}^n $ (or on $\mathbb{X}^n$).  Moreover the series 
						\[V(z)=\sum_{k=1}^\infty \epsilon_k \left\vert \left( P^{-1}z\right) ^{\alpha (k)}\right\vert^2\]
						is a common  Lyapunov function on  $\mathbb{D}^n $ (or on $\mathbb{X}^n$).
				\end{corollary}
				
				\begin{proof}
					It follows from Lemma \ref{lemmepreuve2}, Corollary \ref{corol:jacobian_solvable}, and Remark \ref{rem:assmp2} that Assumptions \ref{assumption:general}-\ref{assumption:lower_triangular} hold. Moreover, it is clear that $\langle L_{\widehat{F}^{(i)}} e_j, e_j\rangle<0$ for all $j$ (see Remark \ref{rem:eigen}). Then, the result follows from Theorem \ref{thm:principalresult} with the sequence
					\begin{equation}\label{eq:compensate_polynomial}
						\begin{cases}b^{(i)}_{jj}=(1-\xi)\\ b^{(i)}_{jk}=\dfrac{\xi}{2 \widehat{K}^{(i)}} & \textrm{if } j\neq k \textrm{ with }\left\langle  L_{\widehat{F}^{(i)}}\widehat{e}_k,\widehat{e}_j\right\rangle \neq 0 \, \textrm{ or } \left\langle  L_{\widehat{F}^{(i)}}\widehat{e}_j,\widehat{e}_k\right\rangle \neq 0 \\
							b^{(i)}_{jk}=0, & \textrm{if } j\neq k \textrm{ with } \left\langle  L_{\widehat{F}^{(i)}}\widehat{e}_k,\widehat{e}_j\right\rangle = 0\, \textrm{ or } \left\langle L_{\widehat{F}^{(i)}}\widehat{e}_j,\widehat{e}_k\right\rangle = 0 , \end{cases}
					\end{equation}
					with $\xi\in]0,1[$. It is clear from \eqref{eq:koop_matrix1_hardy} that, for a fixed $j$ and for all $k\in \mathbb{N}\setminus\{j\}$, there are at most $\widehat{K}^{(i)}$ nonzero values $\langle  L_{\widehat{F}^{(i)}}\widehat{e}_k,\widehat{e}_j\rangle$ and at most $K^{(i)}$ nonzero values $\langle  L_{\widehat{F}^{(i)}}\widehat{e}_j,\widehat{e}_k\rangle$, so that the sequence \eqref{eq:compensate_polynomial} satisfies $\sum_{k=1}^\infty b_{jk}^{(i)} \leq 1$. The elements $Q^{(i)}_{jk}$ of the double sequence \eqref{eq:double_sequence} are given by
					\begin{equation}\label{eq:maxim2_poly}
						Q^{(i)}_{jk}= 
						\dfrac{(\widehat{K}^{(i)})^2 \left| \left\langle  L_{\widehat{F}^{(i)}}\widehat{e}_k,\widehat{e}_j\right\rangle \right|^2}{\xi^2 \, \left|\Re\left(\left\langle L_{\widehat{F}^{(i)}}\widehat{e}_j, \widehat{e}_j \right\rangle\right)\right|\left|\Re\left(\left\langle L_{\widehat{F}^{(i)}} \widehat{e}_k, \widehat{e}_k \right\rangle\right)\right|}.
					\end{equation}
					The condition \eqref{eq:maxim1} implies that $\max_{i=1,\cdots,m} \limsup_{j \in \mathbb{N}} \max_{k=1,\dots,j-1} Q^{(i)}_{jk}\stackrel{\text{def}}{=}Q<1$ for some $\xi\in]0,1[$, so that \eqref{eq:defineepsilon} implies, with the arbitrary choice $\epsilon_1=1$, that
					\begin{equation}
						\label{eq:epsilon_diag1}
						\epsilon_j = \max_{k\in \mathcal{K}_j} \left\{\epsilon_k \, Q \right\} = Q
					\end{equation}
					for $j>1$, with $\mathcal{K}_j=\{k \in\{1,\dots,j-1\} : \left\langle L_{\widehat{F}^{(i)}} \widehat{e}_k,  \widehat{e}_j \right\rangle \neq 0 \textrm{ for some } i \in\{1,\dots,m\} \}$. It follows that \eqref{eq:conditionconvergence_bis} is convergent for any $\rho < 1$ so that Lemma \ref{lemmepreuve2} implies that \eqref{eq:lemmepreuve2_serie1} and \eqref{eq:lemmepreuve2_serie2} are convergent for any $\rho < 1$. Theorem \ref{thm:principalresult} implies that the switched system $\left\lbrace \dot{\widehat{z}} =\widehat{F}^{(i)}(z) \right\rbrace_{i=1}^m$ is GUAS on $S=P^{-1} \left(\mathbb{D}^n\right)\subseteq \mathbb{D}^n $ (or on $S=P^{-1}\left(\mathbb{X}^n\right)\subseteq \mathbb{X}^n$)  and then  the system \eqref{eq:switchonpoly3} is GUAS on $\mathbb{D}^n$ (or on $\mathbb{X}^n$).
					
					In addition, from Theorem \ref{thm:principalresult} we deduce that
						\begin{equation*}
							\widehat{V}(z)=\sum_{k=1}^\infty \epsilon_k \left\vert  \widehat{z}^{\alpha (k)}\right\vert^2
						\end{equation*}
						is a CLF of $\left\lbrace \dot{\widehat{z}} =\widehat{F}^{(i)}(z) \right\rbrace_{i=1}^m$ on $S$.
						If we define $V=\widehat{V}\circ P^{-1}:\mathbb{D}^n\rightarrow \mathbb{R}^+$ (or $\mathbb{X}^n\rightarrow \mathbb{R}^+$), we have 
						the CLF 
						\begin{equation*}
							V(z)=\sum_{k=1}^\infty \epsilon_k \left\vert \left(P^{-1}z\right)^{\alpha (k)}\right\vert^2
						\end{equation*}
						for the switched nonlinear system \eqref{eq:switchonpoly3}.
				\end{proof}
				
				\begin{remark}\label{remark:estimate_basin_attraction_1}
					We can use the (truncated) series
					\begin{equation}\label{clf_examples_1}
						V(z)=\sum_{k= 1}^{k_{max}} \epsilon_k\left|\left(P^{-1}z\right)^{\alpha(k)}\right|^{2}
					\end{equation}
					(where, for example, $\epsilon_1=1$ and $\epsilon_k=Q$ for $k>1$) obtained in Corollary \ref{coroll:poly} as a candidate CLF in order to obtain a numerical estimation of the region of attraction (possibly larger than the unit polydisk). This will be illustrated in Section \ref{sec:illustra}. In this case, the invariance of the flow over the polydisk (or hypercube) is not required. 
				\end{remark}

				\subsubsection{Global uniform asymptotic stability for analytic vector fields}
				
				Another result is obtained when a diagonal dominance property is assumed for the Jacobian matrices $JF^{(i)}(0)$.
				\begin{corollary}\label{coroll:gen2}
					Let 
					\begin{equation}\label{eq:switchonpoly4}
						\left\lbrace \dot z =F^{(i)}(z) \right\rbrace_{i=1}^m
					\end{equation}
					be a switched nonlinear system on $\mathbb{D}^n$ (or on $\mathbb{X}^n$), with $F_\ell^{(i)}(z)=\sum_{k=1}^{+\infty}a^{(i)}_{\ell,k}z^{\alpha(k)}$ and $\sum_{k=1}^{+\infty} \vert a^{(i)}_{\ell,k}\vert  <\infty$ for all $i=\{1,\dots,m\}$ and $l\in\{1,\dots,n\}$. Assume that
					\begin{itemize}
						\item all subsystems of \eqref{eq:switchonpoly4} have a common hyperbolic equilibrium $z_e=0$ that is globally asymptotically stable on $\mathbb{D}^n$ (or on $\mathbb{X}^n$),
						\item the Lie algebra $\textrm{span}\left\lbrace JF^{(i)}(0)\right\rbrace_{Lie} $ is solvable (and therefore there exists a matrix $P$ such that $PJF^{(i)}(0)P^{-1}$ are upper triangular),
						\item there exists $\rho \in]0,1]$ such that $\mathbb{D}^n\left(\rho\right)$ (or $\mathbb{X}^n(\rho)$) is forward invariant with respect to the flows $\varphi_t^{(i)}$ generated by $F^{(i)}$.
					\end{itemize}
					
					If there exist $\xi\in]0,1[$ and $\kappa\in]0,1[$ with $\xi+\kappa<1$ such that, for all $q, r\in \{1,\cdots, n\} $ with $q<r$ (when $n>1$),
					\begin{equation}
						\label{eq:diagonal_dominant}
						\left|J\widehat{F}^{(i)}(0)]_{qr}\right|^2 < \left(\frac{2\xi}{n^2-n}\right)^2 \left|\Re([J\widehat{F}^{(i)}(0)]_{rr})\right| \left| \Re([J\widehat{F}^{(i)}(0)]_{qq}) \right|,
					\end{equation}
					\begin{equation}
						\label{eq:diagonal_dominant2}
						\left|[J\widehat{F}^{(i)}(0)]_{qr}\right| < \frac{2\xi}{n^2-n} \left|\Re([J\widehat{F}^{(i)}(0)]_{qq})\right|
					\end{equation}
					
					and
					\begin{equation}
						\label{eq:cond_corol2}
						\max_{i=1,\cdots,m} \limsup_{j \in \mathbb{N}} \max_{\substack{k=1,\dots,j-1 \\ \left\langle L_{\widehat{F}^{(i)}} \widehat{e}_k,  \widehat{e}_j \right\rangle\neq 0}} \dfrac{\sum_{\ell=1}^{\infty}\left| \left\langle  L_{\widehat{F}^{(i)}}\widehat{e}_\ell,\widehat{e}_j\right\rangle \right| \sum_{\ell=1}^{\infty}\left| \left\langle  L_{\widehat{F}^{(i)}}\widehat{e}_k,\widehat{e}_\ell\right\rangle \right|}{\kappa^2 \, \left|\Re\left(\left\langle  L_{\widehat{F}^{(i)}} \widehat{e}_j, \widehat{e}_j \right\rangle\right)\right|\left|\Re\left(\left\langle   L_{\widehat{F}^{(i)}} \widehat{e}_k,  \widehat{e}_k \right\rangle\right)\right|}<\frac{1}{\rho^2 },
					\end{equation}
					where $\widehat{e}_j(\widehat{z})=\widehat{z}^{\alpha(j)}$ are monomials in the new coordinates $\widehat{z}=P^{-1}z$, then \eqref{eq:switchonpoly4} is GUAS on  $\mathbb{D}^n(\rho)$ (or on $\mathbb{X}^n(\rho)$). 
					Moreover the series 
						\[V(z)=\sum_{k=1}^\infty \epsilon_k \left\vert \left( P^{-1}z\right) ^{\alpha (k)}\right\vert^2\]
						is a common  Lyapunov function on $ \mathbb{D}^n (\rho)$ (or on $\mathbb{X}^n(\rho)$).
				\end{corollary}
				
				\begin{proof}
					We will denote by $N\stackrel{\text{def}}{=} (n^2-n)/2$ the number of upper off-diagonal entries of the Jacobian matrices $J\widehat{F}^{(i)}(0)$.
					It follows from Lemma \ref{lemmepreuve2}, Corollary \ref{corol:jacobian_solvable}, and Remark \ref{rem:assmp2} that Assumptions \ref{assumption:general}-\ref{assumption:lower_triangular} hold. Moreover, it is clear that $\langle L_{\widehat{F}^{(i)}} e_j, e_j\rangle<0$ for all $j$ and $i$ (see Remark \ref{rem:eigen}).
					The result follows from Theorem \ref{thm:principalresult} with the sequence
					\begin{equation}\label{eq:compensate_analytic}
						\begin{cases}b^{(i)}_{jj}=(1-\xi-\kappa)\\ b^{(i)}_{jk}=\dfrac{\xi}{2N} & \textrm{if } j\neq k \textrm{ with }  |\alpha(j)|=|\alpha(k)|  \textrm{, and if }\left\langle  L_{\widehat{F}^{(i)}}\widehat{e}_k,\widehat{e}_j\right\rangle \neq 0 \, \textrm{ or } \left\langle  L_{\widehat{F}^{(i)}}\widehat{e}_j,\widehat{e}_k\right\rangle \neq 0 \\
							b^{(i)}_{jk}=0 & \textrm{if } |\alpha(j)|=|\alpha(k)|, \left\langle  L_{\widehat{F}^{(i)}}\widehat{e}_k,\widehat{e}_j\right\rangle = 0 \, \textrm{ and } \left\langle  L_{\widehat{F}^{(i)}}\widehat{e}_j,\widehat{e}_k\right\rangle = 0 \\
							b^{(i)}_{jk}=\dfrac{\kappa}{2} \dfrac{\left| \left\langle L_{\widehat{F}^{(i)}} \widehat{e}_k,\widehat{e}_j\right\rangle \right|}{\sum_{\ell=1}^{\infty}\left| \left\langle  L_{\widehat{F}^{(i)}}\widehat{e}_\ell,\widehat{e}_j\right\rangle \right|} & \mbox{if}\,  |\alpha(k)|< |\alpha(j)|\\
							b^{(i)}_{jk}=\dfrac{\kappa}{2} \dfrac{\left| \left\langle  L_{\widehat{F}^{(i)}} \widehat{e}_j, \widehat{e}_k\right\rangle \right|}{\sum_{\ell=1}^{\infty}\left| \left\langle L_{\widehat{F}^{(i)}}\widehat{e}_j,\widehat{e}_\ell \right\rangle \right|} & \mbox{if}\,   |\alpha(k)|> |\alpha(j)| \end{cases}
					\end{equation}
					with $\xi\in ]0,1[$ and $\kappa\in ]0,1[$. It follows from \eqref{eq:total_degree_bloc_hardy} in Remark \ref{rem:linear_vector_field} and the fact that the Jacobian matrices $J\widehat{F}^{(i)}(0)$ are upper triangular that, for a fixed $j$ and all $k \neq j$ with $|\alpha(k)|= |\alpha(j)|$, there are at most $N$ nonzero values $\left\langle  L_{\widehat{F}^{(i)}}\widehat{e}_k,\widehat{e}_j\right\rangle$ and at most $N$ nonzero values $\left\langle  L_{\widehat{F}^{(i)}}\widehat{e}_j,\widehat{e}_k\right\rangle$. Therefore, the sequence $b^{(i)}_{jk}$ satisfies 
					\begin{equation*}
						\sum_{k=1}^\infty b^{(i)}_{jk} <(1-\xi-\kappa)+ \xi + \dfrac{\kappa}{2}    \dfrac{\sum_{k=1}^j \left| \left\langle L_{\widehat{F}^{(i)}} \widehat{e}_k,\widehat{e}_j\right\rangle \right|}{\sum_{\ell=1}^{\infty}\left| \left\langle  L_{\widehat{F}^{(i)}}\widehat{e}_\ell,\widehat{e}_j\right\rangle \right|} + \dfrac{\kappa}{2}    \dfrac{\sum_{k=j+1}^\infty \left| \left\langle L_{\widehat{F}^{(i)}} \widehat{e}_j,\widehat{e}_k\right\rangle \right|}{\sum_{\ell=1}^{\infty}\left| \left\langle  L_{\widehat{F}^{(i)}}\widehat{e}_j,\widehat{e}_\ell\right\rangle \right|} < 1.
					\end{equation*}
					The elements $Q^{(i)}_{jk}$ of the double sequence \eqref{eq:double_sequence} are given by
					\begin{equation}\label{eq:maxim2_analytic}
						Q^{(i)}_{jk}= \begin{cases}
							\dfrac{D^2\left| \left\langle  L_{\widehat{F}^{(i)}}\widehat{e}_k,\widehat{e}_j\right\rangle \right|^2}{\xi^2 \, \left|\Re\left(\left\langle L_{\widehat{F}^{(i)}}\widehat{e}_j, \widehat{e}_j \right\rangle\right)\right|\left|\Re\left(\left\langle L_{\widehat{F}^{(i)}} \widehat{e}_k, \widehat{e}_k \right\rangle\right)\right|}&\, \textrm{if } |\alpha(j)|=|\alpha(k)|  \\
							\dfrac{\sum_{\ell=1}^{\infty}\left| \left\langle  L_{\widehat{F}^{(i)}}\widehat{e}_\ell,\widehat{e}_j\right\rangle \right| \sum_{\ell=1}^{\infty}\left| \left\langle  L_{\widehat{F}^{(i)}}\widehat{e}_k,\widehat{e}_\ell\right\rangle \right|}{\kappa^2\left|\Re\left(\left\langle  L_{\widehat{F}^{(i)}} \widehat{e}_j, \widehat{e}_j \right\rangle\right)\right|\left|\Re\left(\left\langle   L_{\widehat{F}^{(i)}} \widehat{e}_k,  \widehat{e}_k \right\rangle\right)\right|} & \textrm{if } |\alpha(k)|\neq |\alpha(j)| \, \textrm{ and }\left\langle L_{\widehat{F}^{(i)}} \widehat{e}_k,  \widehat{e}_j \right\rangle\neq 0 \\
							0 & \textrm{otherwise.}
						\end{cases}
					\end{equation}
					We note that $\sum_{\ell=1}^{\infty}\left| \left\langle L_{\widehat{F}^{(i)}}\widehat{e}_\ell, \widehat{e}_j \right\rangle \right| $ and $ \sum_{\ell=1}^{\infty}\left| \left\langle  L_{\widehat{F}^{(i)}}\widehat{e}_k,\widehat{e}_\ell\right\rangle \right|$ are finite according to the assumption. 
					
					Next, we show that the conditions \eqref{eq:diagonal_dominant} and \eqref{eq:diagonal_dominant2} imply that $Q^{(i)}_{jk}<1$ if $|\alpha(j)| = |\alpha(k)|$. Indeed, it follows from \eqref{eq:total_degree_bloc_hardy} and \eqref{eq:maxim2_analytic} that this latter inequality is equivalent to
					\begin{equation*}
						\alpha^2_q(k) |[J\widehat{F}^{(i)}(0)]_{qr}|^2 < \frac{\xi^2}{N^2} \left|\sum_{\ell=1}^n \alpha_\ell(j) \Re([J\widehat{F}^{(i)}(0)]_{\ell\ell}) \right| \left| \sum_{\ell=1}^n \alpha_\ell(k) \Re([J\widehat{F}^{(i)}(0)]_{\ell\ell}) \right|
					\end{equation*}
					for all $j>k$ such that $\alpha(j)=(\alpha_1(k),\cdots, \alpha_{q}(k)-1,\cdots, \alpha_{r}(k)+1, \cdots,\alpha_n(k))$ for some $q<r$. Since the diagonal entries of the (upper-triangular) Jacobian matrices $J\widehat{F}^{(i)}(0)$ are the eigenvalues and therefore have negative real parts, the most restrictive case is obtained with $\alpha_\ell(k)=0$ for all $l\neq q$, which yields
					\begin{equation*}
						\alpha^2_q(k) |[J\widehat{F}^{(i)}(0)]_{qr}|^2 < \frac{\xi^2}{N^2} \left| (\alpha_q(k)-1)  \Re([J\widehat{F}^{(i)}(0)]_{qq})+ \Re([J\widehat{F}^{(i)}(0)]_{rr})\right| \left| \alpha_q(k) \Re([J\widehat{F}^{(i)}(0)]_{qq}) \right|.
					\end{equation*}
					When $\alpha_q(k)=1$, this inequality is equivalent to \eqref{eq:diagonal_dominant}. When $\alpha_q(k)>1$, we can rewrite
					\begin{equation*}
						\begin{split}
							(\alpha_q(k)-1) & |[J\widehat{F}^{(i)}(0)]_{qr}|^2 +  |[J\widehat{F}^{(i)}(0)]_{qr}|^2 \\
							&< \frac{\xi^2}{N^2} \left( (\alpha_q(k)-1)  \left|\Re([J\widehat{F}^{(i)}(0)]_{qq})\right|^2+ \left| \Re([J\widehat{F}^{(i)}(0)]_{rr})\right| \left| \Re([J\widehat{F}^{(i)}(0)]_{qq}) \right| \right).
						\end{split}
					\end{equation*}
					Using \eqref{eq:diagonal_dominant}, we have that the above inequality is satisfied if
					\begin{equation*}
						(\alpha_q(k)-1) |[J\widehat{F}^{(i)}(0)]_{qr}|^2 < \frac{\xi^2}{N^2} (\alpha_q(k)-1)  \left|\Re([J\widehat{F}^{(i)}(0)]_{qq})\right|^2,
					\end{equation*}
					which is equivalent to \eqref{eq:diagonal_dominant2}.
					
					While $Q^{(i)}_{jk}<1$ for $|\alpha(j)| = |\alpha(k)|$, it is easy to see that $Q^{(i)}_{jk}>1$ for $|\alpha(j)| > |\alpha(k)|$. The condition \eqref{eq:cond_corol2} therefore implies that 
					$\max_{i=1,\cdots,m} \limsup_{j \in \mathbb{N}} \max_{k=1,\dots,j-1} Q^{(i)}_{jk} \stackrel{\text{def}}{=}Q < 1/\rho^2$ and \eqref{eq:defineepsilon} yields
					\begin{equation}
						\label{eq:epsilon_diag}
						\epsilon_j \sim \max_{k\in \mathcal{K}_j} \left\{\epsilon_k \, Q \right\}
					\end{equation}
					for $j\gg 1$, with 
					$$\mathcal{K}_j=\{k \in{1,\dots,j-1} : \left\langle L_{\widehat{F}^{(i)}} \widehat{e}_k,  \widehat{e}_j \right\rangle \neq 0 \textrm{ for some } i \in\{1,\dots,m\}  \textrm{ and } |\alpha(k)| < |\alpha(j)|\}.$$
					Hence, the sequence \eqref{eq:epsilon_diag} leads to $\epsilon_j=\mathcal{O}(Q^{|\alpha(j)|})$. It follows that \eqref{eq:conditionconvergence_bis} is convergent on $S=P^{-1} \left(\mathbb{D}^n\right)\subseteq \mathbb{D}^n $ (or on $S=P^{-1}\left(\mathbb{X}^n\right)\subseteq \mathbb{X}^n$) and Lemma \ref{lemmepreuve2} implies that \eqref{eq:lemmepreuve2_serie1} and \eqref{eq:lemmepreuve2_serie2} are convergent on $S$. Theorem \ref{thm:principalresult} implies that the switched  system  $\left\lbrace \dot{\widehat{z}} =\widehat{F}^{(i)}(z) \right\rbrace_{i=1}^m$ is GUAS on $S$  and then the system \eqref{eq:switchonpoly4} is GUAS on $\mathbb{D}^n(\rho)$ (or on $\mathbb{X}^n(\rho)$).
					
				In addition, from Theorem \ref{thm:principalresult} we deduce that
						\begin{equation*}
							\widehat{V}(z)=\sum_{k=1}^\infty \epsilon_k \left\vert  \widehat{z}^{\alpha (k)}\right\vert^2
						\end{equation*}
						is a CLF of $\left\lbrace \dot{\widehat{z}} =\widehat{F}^{(i)}(z) \right\rbrace_{i=1}^m$ on $S$.
						If we define $V=\widehat{V}\circ P^{-1}:\mathbb{D}^n\left( \rho\right)\rightarrow \mathbb{R}^+$ (or $\mathbb{X}^n\left(\rho\right)\rightarrow \mathbb{R}^+$), we have 
						the CLF 
						\begin{equation*}
							V(z)=\sum_{k=1}^\infty \epsilon_k \left\vert \left(P^{-1}z\right)^{\alpha (k)}\right\vert^2
						\end{equation*}
						for the switched nonlinear system \eqref{eq:switchonpoly4}.
				\end{proof}
			
					\begin{remark}\label{rem:weights_2}
						The results of Corollaries \ref{coroll:poly} and \ref{coroll:gen2} could potentially be extended to other RKHSs of analytic functions, with a basis of weighted monomials. To this end, the sequences $b^{(i)}_{jk}$ should be adapted to the weighting coefficients $\omega_\alpha$ of the monomials, while still satisfying the conditions requested in Theorem \ref{thm:principalresult}.
					\end{remark}

				For the particular case where the Jacobian matrices $JF^{(i)}(0)$ are simultaneously diagonalizable (i.e. they are diagonalizable and they commute), the diagonal dominance conditions \eqref{eq:diagonal_dominant} and \eqref{eq:diagonal_dominant2} are trivially satisfied. We should mention that the Lie-algebraic property of commutation is only needed for the Jacobian matrices $JF^{(i)}(0)$, an assumption which contrasts with the commutation property imposed on vector fields in \cite{mancilla2000condition}, \cite{Shim2001CommonLF} and \cite{vu2005common}.
				
				\begin{remark}\label{remark:estimate_basin_attraction_2}
					Similarly to Remark \ref{remark:estimate_basin_attraction_1}, we can use the (truncated) series
					\begin{equation}\label{clf_examples_3}
						V(z)=\sum_{k= 1}^{k_{max}} Q^{2|\alpha(k)|}\left|\left(P^{-1}z\right)^{\alpha(k)}\right|^{2}
					\end{equation}
					obtained in Corollary \ref{coroll:gen2} as a candidate CLF in order to estimate the region of attraction. Again, the invariance of the flow over the polydisk (or hypercube) is not required.
				\end{remark}
				
				\begin{remark}\label{rem:nonpoly_dim1}
					In the case $n=1$, we recover the trivial GUAS property of switched systems from Corollary \ref{coroll:gen2}. Indeed, consider the vector fields $F^{(i)}(z)= \sum_{k= 1}^{\infty} a^{(i)}_{k} z^k$ on $\mathbb{D}$ (or on $]-1,1[$), with $\sum_{k= 1}^{\infty} \vert a^{(i)}_{k}\vert <\infty$, and assume that the subsystems have a globally stable equilibrium at the origin. 
					The Lie-algebra generated by the scalars $JF^{(i)}(0)$ is trivially solvable and $\mathbb{D}(\rho)$ (or on $]-\rho,\rho[$) is forward invariant for all $\rho$. Moreover, the conditions \eqref{eq:diagonal_dominant} and \eqref{eq:diagonal_dominant2} are trivially satisfied. Then Corollary \ref{coroll:gen2} implies that the switched system is GUAS on $\mathbb{D}(\rho)$  (or on $]-\rho,\rho[$) for $\rho \in ]0,1]$ which satisfies \eqref{eq:cond_corol2}.
					It follows from \eqref{eq:koop_matrix1_hardy} that
					\[\sum_{\ell=1}^{\infty}\left| \left\langle L_{F^{(i)}}e_\ell, e_j\right\rangle \right| =\sum_{\ell=1}^j \ell\vert a^{(i)}_{{j-\ell+1}}\vert=\sum_{\ell=1}^j (j-\ell+1)\vert a^{(i)}_{\ell} \vert,\]
					\[\sum_{\ell=1}^{\infty}\left| \left\langle  L_{F^{(i)}}e_k,e_\ell\right\rangle \right| =k\sum_{\ell=1}^\infty \vert a^{(i)}_{\ell}\vert,\]
					and $\left| \Re\left(\left\langle e_j,  L_{F^{(i)}}e_j\right\rangle \right) \right|=j\left| \Re (a^{(i)}_{1})\right| $.
					With $\kappa$ arbitrarily close to $1$ (since $\xi$ can be taken arbitrarily small in \eqref{eq:diagonal_dominant} and \eqref{eq:diagonal_dominant2}), condition \eqref{eq:cond_corol2} is rewritten as
					\begin{equation*}
						\max_{i=1,\cdots,m} \limsup_{j \in \mathbb{N}} \frac{\sum_{\ell= 1}^{j} (j-\ell+1) \vert a^{(i)}_{\ell}\vert   \sum_{\ell= 1}^{\infty} \vert a^{(i)}_{\ell} \vert }{j\left|\Re(a^{(i)}_{1})\right|^2 } < \frac{1}{\rho^2}
					\end{equation*}
					and, using $(j-\ell+1) \vert a^{(i)}_{\ell}\vert \leq j \vert a^{(i)}_{\ell}\vert $ for all $\ell\leq j$, we obtain
					\begin{equation*}
						\rho < \min_{i=1,\cdots,m} \frac{\sum_{\ell= 1}^{\infty} \vert a^{(i)}_{\ell}\vert}{\left|\Re(a^{(i)}_{1})\right|}.
					\end{equation*}
				\end{remark}

				\section{Examples}
				\label{sec:illustra}
				
				This section presents two examples that illustrate our results. We will focus on specific cases that satisfy the assumptions of Corollaries \ref{coroll:poly} and \ref{coroll:gen2} and, without loss of generality, we will directly consider Jacobian matrices in triangular form.

				\subsection{Illustration 1: polynomial vector fields }
				
				Consider the vector fields on the square $\mathbb{X}^2$
				
				\begin{equation}
					\label{eq:illustration_1}
					\begin{array}{rcl}
						F^{(1)}(x_1,x_2) & = & \left(-x_1, - x_2-0.1x_1^3 \right) \\
						F^{(2)}(x_1,x_2) & = & \left(-x_1,- x_2-0.1x_1^4\right).
					\end{array}
				\end{equation}
				
			For both subsystems, the origin is the unique equilibrium and it is globally asymptotically stable.
			The square $\mathbb{X}^2$ is invariant with respect to the flows of $F^{(i)}$. Indeed, for all $x\in \partial \mathbb{X}^2$ (i.e. $|x_\ell|=1$ for some $l$), one has to verify that $F^{(i)}_\ell(x) \, x_\ell \leq 0$. We have
			\begin{itemize}
				\item $|x_1|=1 \Rightarrow  F^{(1)}_1(x) x_1=-x_1^2 =-1< 0$,
				\item $|x_2|=1 \Rightarrow F^{(1)}_2(x) x_2 =-x_2^2 -0.1x_1^3x_2=-1-0.1x_1^3x_2< 0$, since $|x_1|^3<1$.
			\end{itemize}
			The same result follows for $F^{(2)}$.
			
			According to \eqref{eq:koop_matrix1_hardy},
			the entries of the Koopman matrices  $\bar{L}_{F^{(1)}}$ and $\bar{L}_{F^{(2)}}$ are given by 
			\[\langle  L_{F^{(1)}}e_k, e_j \rangle= \begin{cases} -\left|\alpha(j)\right| & \mbox{if } k=j 
				\\
				-0.1\left(\alpha_2(j)+1\right) & \mbox{if } \alpha_1(k) =\alpha_1(j)-3\geq 0 \mbox{ and }  \alpha_2(k)=\alpha_2(j)+1\\
				0& \mbox{otherwise }
			\end{cases}
			\]
			and 
			\[\langle  L_{F^{(2)}}e_k, e_j \rangle= \begin{cases} -\left|\alpha(j)\right|  & \mbox{if } k=j 
				\\
				-0.1\left(\alpha_2(j)+1\right) & \mbox{if } \alpha_1(k) =\alpha_1(j)-4\geq 0 \mbox{ and }  \alpha_2(k)=\alpha_2(j)+1\\
				0& \mbox{otherwise }.
			\end{cases}
			\]
			Moreover, $K^{(1)}=K^{(2)}=1$ and the condition \eqref{eq:maxim1} can be rewritten as
			\[Q=\limsup_{\substack{j \in \mathbb{N} \\ |\alpha(j)|>4}} \max \left\{\dfrac{0.01\left(\alpha_1(j)+1\right)^2}{\left|\alpha(j)\right|\left(\left|\alpha(j)\right|-3\right)  }, \dfrac{0.01\left(\alpha_2(j)+1\right)^2 }{\left|\alpha(j)\right|  \left(\left|\alpha(j)\right|-4\right)  } \right\}=0.01.
			\]
			Hence, it follows from Corollary \ref{coroll:poly} that the switched system \eqref{eq:illustration_1} is GUAS in $\mathbb{X}^2$. 
			As explained in Remark \ref{remark:estimate_basin_attraction_1}, we have the following candidate CLF 
			\begin{equation}\label{clf_example_2}
				V(x)=\sum_{k= 1}^{k_{max}} \epsilon_k\left|x^{\alpha(k)}\right|^{2},
			\end{equation}
			where $\epsilon_1=1$ and $\epsilon_k=Q$ for all $k>1$.
			The largest level set of $V$ lying in the region where $\dot{V}<0$ provides an inner approximation of the region of attraction of the origin. As shown in Figure \ref{fig:compare_1}, the approximation is larger than the one obtained with a quadratic CLF 
			computed for the linearized switched system.
			
			\begin{figure}[ht]
				\begin{center}
					\includegraphics[scale=0.55]{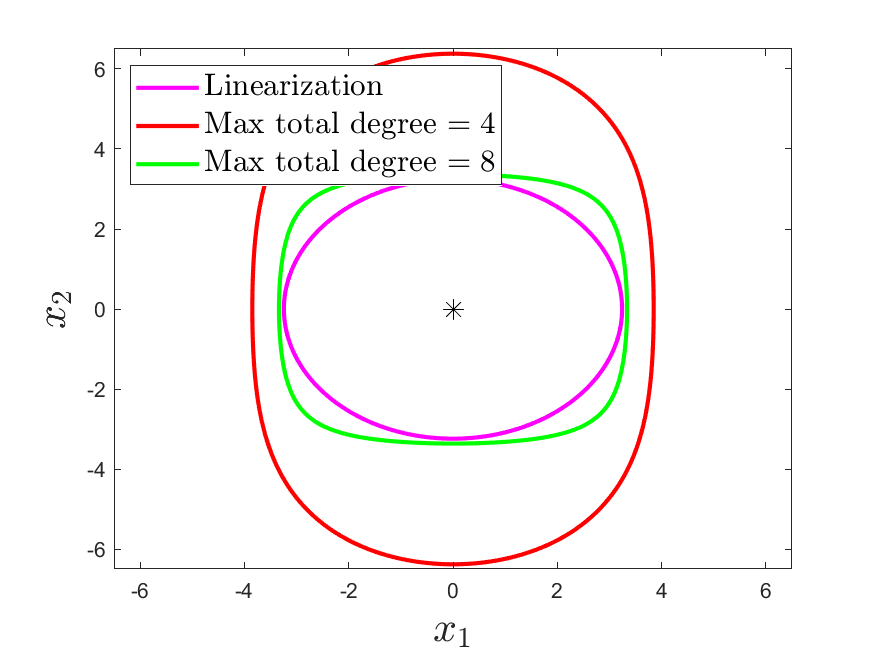}
					\caption{For the switched system \eqref{eq:illustration_1}, the region of attraction is estimated with the quadratic CLF $V(x_1,x_2)=x_1^2+x_2^2$ (magenta curve), with the CLF \eqref{clf_example_2} truncated to the total degree 4 (i.e. $k_{max}=5$) (red curve), and with \eqref{clf_example_2} truncated to the total degree 8 (i.e. $k_{max}=14$) (green curve).}
					\label{fig:compare_1}
				\end{center}
			\end{figure}

			\subsection{Illustration 2: analytic vector fields} The following example is inspired from \cite{angeli2000note} and \cite{liberzon2003switching}.

			Consider the switched system defined by the vector fields 
			\begin{equation}
				\label{eq:illust2}
				\begin{array}{rcl}
					F^{(1)}(x_1,x_2) & = & \left(-  x_1 +\eta \sin^2(x_1) x_2^2,- x_2 \right) \\
					F^{(2)}(x_1,x_2) & = & \left(-  x_1 +\eta\cos^2(x_1) x_2^2,-x_2\right),
				\end{array}
			\end{equation}
			with $\eta\in ]0, 1]$. 
			Both subsystems possess an equilibrium at the origin, which is globally asymptotically stable in $\mathbb{R}^2$ (see similar arguments as in \cite{liberzon2003switching}). For all $\rho\leq 1$, the square $\mathbb{X}^2(\rho)$ is invariant with respect to the flows of $F^{(i)}$. Indeed, we have
			\begin{itemize}
				\item $|x_1|=\rho \Rightarrow F^{(1)}_1(x)x_1=-\rho+\eta x_1x_2^2\sin^2(x_1)< 0$ since
				\[\left|\eta x_1x_2^2\sin^2(x_1)\right|  \leq \eta \rho\left|x_2^2\right| <\eta \rho^3\leq \rho.\]
				\item $|x_2|=1 \Rightarrow F^{(1)}_2(x)x_2=-\rho<0$.
			\end{itemize}
			The same result follows for $F^{(2)}$.
			The Taylor expansion of the vector fields yields
			\begin{eqnarray*}
				F^{(1)}(x) & = & \left(-x_1 + \eta \sum_{p=1}^\infty\dfrac{(-1)^{p+1}2^{2p-1}}{(2p)!}x_1^{2p} x_2^2,-x_2\right) \\
				F^{(2)}(x) & = & \left(-x_1 + \eta x_2^2+ \eta \sum_{p=1}^\infty\dfrac{(-1)^{p}2^{2p-1}}{(2p)!}x_1^{2p} x_2^2,-x_2 \right).
			\end{eqnarray*}
			According to \eqref{eq:koop_matrix1_hardy}, the entries of the Koopman matrices $\bar{L}_{F^{(1)}}$ and $\bar{L}_{F^{(2)}}$ are given by
			\[\langle  L_{F^{(1)}}e_k, e_j \rangle= \begin{cases} -\left|\alpha(k)\right|  & \mbox{if } k=j \\
				\eta \alpha_1(k)\,  \dfrac{(-1)^{p+1}2^{2p-1}}{(2p)!}  & \mbox{if } \alpha_1(k)=\alpha_1(j)+1-2p\geq 0 \text{ and }\alpha_2(k) =\alpha_2(j)-2\geq 0 \\ 
				0& \mbox{otherwise }
			\end{cases}
			\]
			and
			\[\langle  L_{F^{(2)}}e_k, e_j \rangle= \begin{cases} -\left|\alpha(k)\right|  & \mbox{if } k=j \\
				\eta \alpha_1(k)   & \mbox{if } \alpha_1(k)=\alpha_1(j)+1 \text{ and }\alpha_2(k) =\alpha_2(j)-2\geq 0 \\ 
				\eta \alpha_1(k) \, \dfrac{(-1)^{p}2^{2p-1}}{(2p)!}  & \mbox{if } \alpha_1(k)=\alpha_1(j)+1-2p\geq 0 \text{ and }\alpha_2(k) =\alpha_2(j)-2\geq 0 \\ 
				0& \mbox{otherwise }
			\end{cases}
			\]
			where $ p=1,\cdots ,  \left\lfloor \frac{\alpha_1(j)-1}{2} \right\rfloor$.
			This implies that we have 
			\begin{equation}
				\label{eq:term_illustrat2_1}
				\begin{array}{l}
					\sum_{\ell=1}^{\infty}\left| \left\langle  L_{F^{(1)}}e_\ell, e_j\right\rangle \right| = \left|\alpha(j)\right|+ \eta \sum_{\ell=1}^{  \lfloor\frac{\alpha_1(j)-1}{2}\rfloor}  \dfrac{(\alpha_1(j)+1-2l) 2^{2l-1}}{(2l)!} \\
					\sum_{\ell=1}^{\infty}\left| \left\langle L_{F^{(2)}}e_\ell, e_j\right\rangle \right| 
					= \left|\alpha(j)\right|+\eta \left(\alpha_1(j)+1\right)+\eta\sum_{\ell=1}^{ \lfloor\frac{\alpha_1(j)-1}{2}\rfloor} \dfrac{(\alpha_1(j)+1-2l) 2^{2l-1}}{(2l)!} \\
					\sum_{\ell=1}^{\infty}\left| \left\langle  L_{F^{(1)}}e_{k},e_\ell\right\rangle \right| = \left|\alpha(k)\right| + \eta \alpha_1(k) \sum_{p=1}^\infty \dfrac{2^{2p-1}}{(2p)!} = \left|\alpha(k)\right| + \eta \alpha_1(k) \frac{\cosh(2)-1}{2} \\
					\sum_{\ell=1}^{\infty}\left| \left\langle  L_{F^{(2)}}e_{k},e_\ell\right\rangle \right|= \left|\alpha(k)\right|  + \eta \alpha_1(k) \left(1+  \sum_{p=1}^\infty \dfrac{2^{2p-1}}{(2p)!}  \right)=\left|\alpha(k)\right| + \eta \alpha_1(k)\frac{\cosh(2)+1}{2}.
				\end{array}
			\end{equation}
			Since the Jacobian matrices $JF^{(i)}(0)$ are diagonal, the conditions \eqref{eq:diagonal_dominant} and \eqref{eq:diagonal_dominant2} are trivially satisfied (with $\xi$ arbitrarily small). Moreover, we observe from \eqref{eq:term_illustrat2_1} that $\sum_{\ell=1}^{\infty}\left| \left\langle  L_{F^{(1)}}e_\ell, e_j\right\rangle \right|\leq \sum_{\ell=1}^{\infty}\left| \left\langle  L_{F^{(2)}}e_\ell, e_j\right\rangle \right|$ and $\sum_{\ell=1}^{\infty}\left| \left\langle  L_{F^{(1)}}e_{k},e_\ell\right\rangle \right|\leq \sum_{\ell=1}^{\infty}\left| \left\langle  L_{F^{(2)}}e_{k},e_\ell\right\rangle \right|$ for all $k,j$. It follows that, with $\kappa$ arbitrarily close to $1$, condition \eqref{eq:cond_corol2} can be rewritten as
			\begin{equation*}
				\limsup_{j \in \mathbb{N}} \max_{k=1,\dots,j-1} \dfrac{\sum_{\ell=1}^{\infty}\left| \left\langle  L_{F^{(2)}}e_\ell,e_j\right\rangle \right| \sum_{\ell=1}^{\infty}\left| \left\langle  L_{F^{(2)}}e_k,e_\ell\right\rangle \right|}{\left|\Re\left(\left\langle  L_{F^{(2)}} e_j, e_j \right\rangle\right)\right|\left|\Re\left(\left\langle   L_{F^{(2)}} e_k,  e_k \right\rangle\right)\right|}<\frac{1}{\rho^2 },
			\end{equation*}
			which is verified for $\rho=\left(1+\frac{\eta}{2}  \left( \cosh(2)+1\right)\right)^{-1}$. Indeed,
			from \eqref{eq:term_illustrat2_1}, we have
			\begin{eqnarray*}
				\sum_{\ell=1}^{\infty}\left| \left\langle L_{F^{(2)}}e_\ell, e_j\right\rangle \right| 
				&<&\left|\alpha(j)\right|+\eta \left(\alpha_1(j)+1\right)\left(1+ \sum_{\ell=1}^{\infty}  \dfrac{2^{2l-1}}{(2l)!}\right)\\
				&=&\left|\alpha(j)\right|+ \eta \left(\alpha_1(j)+1\right)\frac{\cosh(2)+1}{2}\\
				&\leq &\left|\alpha(j)\right|\left(1+\frac{\eta}{2}  \left( \cosh(2)+1\right)\right).
			\end{eqnarray*}
			and 
			\begin{equation*}
				\sum_{\ell=1}^{\infty}\left| \left\langle  L_{F^{(2)}}e_{k},e_\ell\right\rangle \right| \leq \left|\alpha(k)\right|\left(1+\frac{\eta}{2}  \left( \cosh(2)+1\right)\right).
			\end{equation*}
			It follows from Corollary \ref{coroll:gen2} that the switched system is GUAS on $\mathbb{X}^2(\rho)$. See Figure \ref{fig:ex3} for the different values of $\rho$ depending on $\eta$.
			
			\begin{figure}[ht]
				\begin{center}
					\includegraphics[scale=0.55]{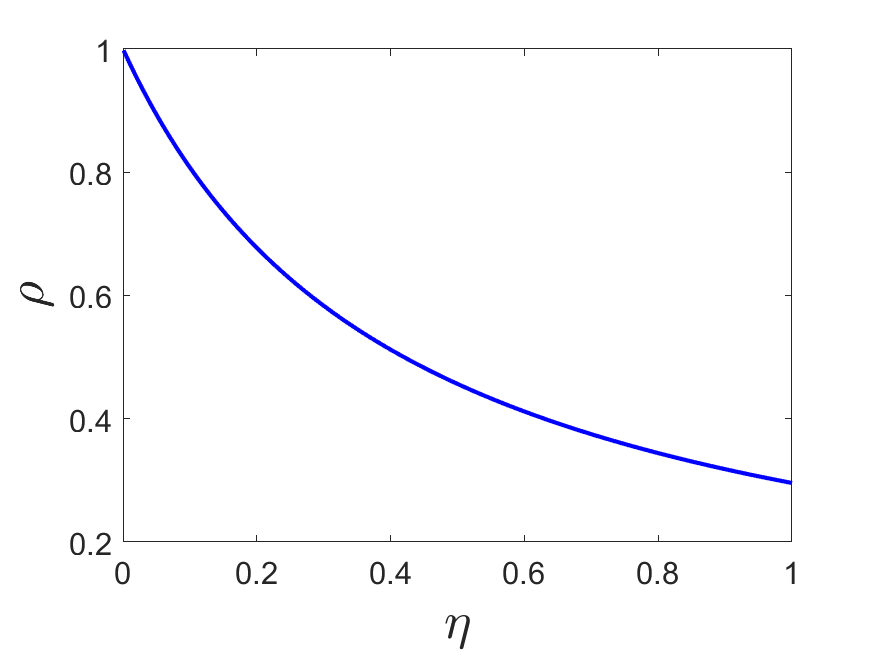} 
					\caption{The switched system is shown to be GUAS on a hypercube of edge length $2\rho$ that depends on the parameter $\eta$.}
					\label{fig:ex3}
				\end{center}
			\end{figure}
			
			As shown in Remark \ref{remark:estimate_basin_attraction_2}, we have the candidate CLF
			\begin{equation}\label{clf_example_3}
				V(x)=\sum_{k=1}^{k_{max}} Q^{2|\alpha(k)|}\left|x^{\alpha(k)}\right|^{2}
			\end{equation}
			where $Q=1/\rho$. The corresponding estimation of the region of attraction is shown in Figure \ref{fig:compare_2} and can be compared to the (smaller) approximation obtained with a quadratic CLF 
			computed for the linearized switched system.
			
			\begin{figure}[ht]
				\begin{center}
					\includegraphics[scale=0.55]{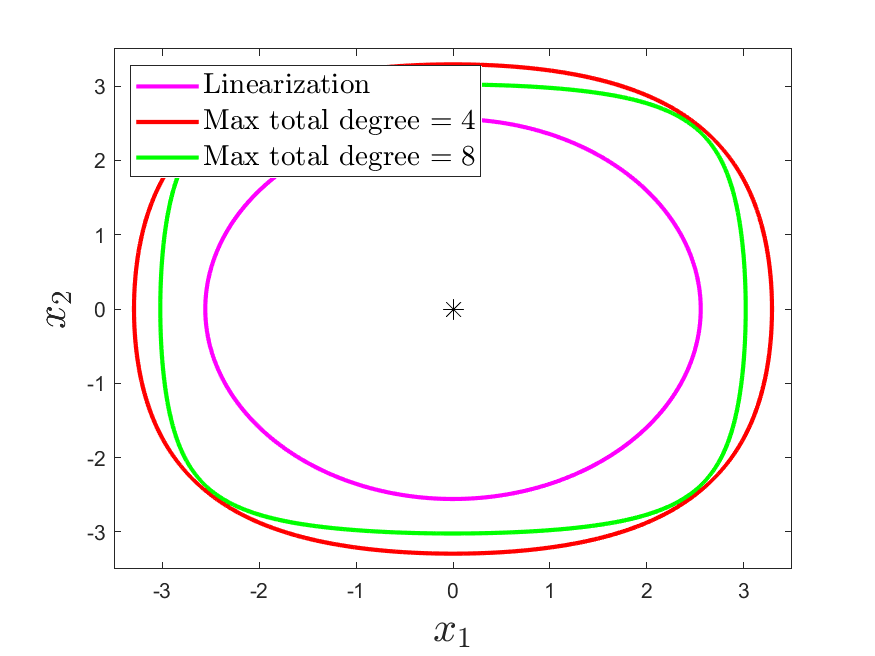}
					\caption{For the switched system \eqref{eq:illust2} with $\eta=1$, the region of attraction is estimated with the quadratic CLF $V(x_1,x_2)=x_1^2+x_2^2$ (magenta curve), with the CLF \eqref{clf_example_3} truncated to the total degree 4 ($k_{max}=5$) (red curve), and with \eqref{clf_example_3} truncated to the total degree 8 ($k_{max}=14$) (green curve).}
					\label{fig:compare_2}
				\end{center}
			\end{figure}

			\section{Conclusion and perspectives}
			\label{sec:concl}
			
			This paper provides new advances on the uniform stability problem for switched nonlinear systems satisfying Lie-algebraic solvability conditions. First, we have shown that the solvability condition on nonlinear vector fields does not guarantee the existence of a common invariant flag and, instead, we have imposed the solvability condition only on the linear part of the vector fields. Then we have constructed a common Lyapunov functional for an equivalent infinite-dimensional  switched linear system obtained with the adjoint of the Koopman  generator on reproducing kernel Hilbert spaces, and more specifically on the Hardy space of the polydisk (or on the real hypercube). Finally we have derived a common Lyapunov function via evaluation functionals to prove that specific switched nonlinear systems are uniformly globally asymptotically stable on invariant sets. Our results heavily rely on the Koopman operator framework, which appears to be a valid tool to tackle theoretical questions from a novel angle.
			
			We envision several perspectives for future research. First of all, our results could be adapted to other spaces of analytic functions with a basis of weighted monomials. Since our main result encompasses general RKHS, we could also make a step further by considering other spaces than spaces of analytic functions and thereby studying other types of (non-analytic) vector fields. In this context, the conservativeness of our general assumptions (such as the existence of an invariant flag) should be investigated. In the same line, our results apply to specific types of switched nonlinear systems within the frame of Lie-algebraic solvability conditions. They could be extended to more general dynamics, including dynamics that possess a limit cycle or a general attractor. In the same line, the Koopman operator-based techniques developed in this paper could be applied to other types of stability than uniform stability. More importantly, the obtained stability results are limited to bounded invariant sets, mainly due to the convergence properties of the Lyapunov functions and the very definition of the Hardy space on the polydisk (or on the real hypercube). We envision that these results could possibly be adapted to infer global stability in $\mathbb{R}^n$. Finally, our results are not restricted to switched systems and have direct implications in the global stability properties of nonlinear dynamical systems, which will be investigated in a future publication.

			\appendix
			\section{General theorems}
			\label{sec:appendix}
			
			We recall here some general results that are used in the proofs of our results.
			\begin{theorem}[Maximum Modulus Principle for bounded domains \cite{scheidemann2005introduction}]\label{maxim_modul}
				Let $D \subset  \mathbb{C}^n$ be a bounded domain and $f : \overline{D} \rightarrow \mathbb{C}$ be  a continuous function, whose restriction to $D$ is holomorphic. Then $|f|$ attains a maximum on the boundary  $\partial D $.
			\end{theorem}
			\begin{theorem}[Abel's multidimensional lemma \cite{scheidemann2005introduction} p.36]\label{abel}
				Let $\sum_{\alpha \in \mathbb{N}^n} a_\alpha z^\alpha$ be a power series. If there exist $r\in \mathbb{C}^n$ such that 
				$\sup_{\alpha \in \mathbb{N}^n} \left| a_\alpha r^\alpha\right|<\infty$, then the series $\sum_{\alpha \in \mathbb{N}^n} a_\alpha z^\alpha$ is  normally convergent for all $z\in \mathbb{C}^n$ such that $|z_1|<|r_1|,\, \cdots , |z_n|<|r_n|$.
			\end{theorem}
			\begin{theorem}[Weierstrass's M-test]\label{wierstrass}
				Let $\sum_{k=1}^{+\infty} f_n(z)$ be a series of functions on a domain $D$ of $\mathbb{C}^n$. If there exists a sequence of real numbers $M_k$ such that 
				\begin{itemize}
					\item $M_k>0$ for all $k$,
					\item the numerical series $\sum_{k=1}^{+\infty}M_k$ is convergent and
					\item $\forall k$, $\forall z\in  D$,  $ \left|f_k(z)\right|\leq M_k$.
				\end{itemize}
				Then the series  $\sum_{k=1}^{+\infty} f_n(z)$ is absolutely and uniformly convergent on  $D$.
			\end{theorem}
			
			\begin{theorem}[Lie's theorem \cite{erdmann2006introduction} p.49]\label{thm:lie}
				Let $X$ be a nonzero $n$-complex vector space,
				and $\mathfrak{g}$ be a solvable Lie subalgebra of the Lie algebra of $n\times n$ complex matrices. Then $X$ has a basis $(v_1,\dots , v_n)$
				with respect to which every element of $\mathfrak{g}$ has an upper triangular form.
			\end{theorem} 
			%
			%

			\bibliographystyle{siamplain}
			
			\bibliography{christian_references}

		\end{document}


\maketitle

\section{A detailed example}

Here we include some equations and theorem-like environments to show
how these are labeled in a supplement and can be referenced from the
main text.
Consider the following equation:
\begin{equation}
  \label{eq:suppa}
  a^2 + b^2 = c^2.
\end{equation}
You can also reference equations such as \cref{eq:matrices,eq:bb} 
from the main article in this supplement.

\lipsum[100-101]

\begin{theorem}
  An example theorem.
\end{theorem}

\lipsum[102]
 
\begin{lemma}
  An example lemma.
\end{lemma}

\lipsum[103-105]

Here is an example citation: \cite{KoMa14}.

\section[Proof of Thm]{Proof of \cref{thm:bigthm}}
\label{sec:proof}

\lipsum[106-112]

\section{Additional experimental results}
\Cref{tab:foo} shows additional
supporting evidence. 

\begin{table}[htbp]
{\footnotesize
  \caption{Example table}  \label{tab:foo}
\begin{center}
  \begin{tabular}{|c|c|c|} \hline
   Species & \bf Mean & \bf Std.~Dev. \\ \hline
    1 & 3.4 & 1.2 \\
    2 & 5.4 & 0.6 \\ \hline
  \end{tabular}
\end{center}
}
\end{table}
\bibliographystyle{siamplain}
\bibliography{christian_references}

%% file: christian_shared.tex
\usepackage{amsfonts}
\usepackage{graphicx}
\usepackage{epstopdf}
\usepackage{algorithmic}
\ifpdf
  \DeclareGraphicsExtensions{.eps,.pdf,.png,.jpg}
\else
  \DeclareGraphicsExtensions{.eps}
\fi

\newsiamremark{remark}{Remark}
\newsiamremark{hypothesis}{Hypothesis}
\crefname{hypothesis}{Hypothesis}{Hypotheses}
\newsiamthm{claim}{Claim}

\headers{Uniform stability of switched nonlinear systems 
}{C. M. Zagabe and A. Mauroy}

\title{Uniform global stability of switched nonlinear systems in the Koopman operator framework%
\thanks{Submitted to the editors
}}

\author{Christian Mugisho Zagabe\thanks{Department of  Mathematics and Namur Research Institute for Complex Systems (naXys),  University of Namur
  (\email{christian.mugisho@unamur.be}),}
\and Alexandre Mauroy \thanks{Department of  Mathematics and Namur Research Institute for Complex Systems (naXys),  University of Namur
  (\email{alexandre.mauroy@unamur.be})}}

\usepackage{amsopn}
